\documentclass[11pt, reqno]{article}

\usepackage{fullpage}
\usepackage{enumerate}
\usepackage{setspace}
\usepackage{listings}
\usepackage{amsmath,amsfonts,amssymb,graphics,amsthm,comment}
\usepackage{hyperref}
\hypersetup{
    colorlinks=true,
    linkcolor=blue,
    citecolor=red,
    urlcolor=blue,
    pdfborder={0 0 0}
} 
\usepackage{dsfont}
\usepackage[T1]{fontenc}
\usepackage[utf8]{inputenc}   

\usepackage[framemethod=TikZ]{mdframed}
\usepackage{lipsum}
\mdfdefinestyle{MyFrame}{%
    linecolor=blue,
    outerlinewidth=2pt,
    roundcorner=20pt,
    innertopmargin=\baselineskip,
    innerbottommargin=\baselineskip,
    innerrightmargin=20pt,
    innerleftmargin=20pt,
    backgroundcolor=gray!50!white}
\usepackage{shadethm} 

\usepackage{graphicx}
\usepackage{xcolor}
\usepackage{wrapfig} 
\usepackage{xcolor}

\usepackage[font=sf, labelfont={sf,bf}, margin=1cm]{caption}

\usepackage{cleveref}
\usepackage{xcolor}

  \Crefname{theorem}{Theorem}{Theorems}
  \Crefname{thm}{Theorem}{Theorems}
  \Crefname{lemma}{Lemma}{Lemmas}
  
  \Crefname{lem}{Lemma}{Lemmas}
  \Crefname{remark}{Remark}{Remarks}
  \Crefname{prop}{Proposition}{Propositions}
  \Crefname{proposition}{Proposition}{Propositions}
  \Crefname{problem}{Problem}{Problems}
\Crefname{notation}{Notation}{Notations}
\Crefname{shadethm}{Theorem}{Theorems}
\Crefname{claim}{Claim}{Claims}
  \Crefname{defn}{Definition}{Definitions}
  \Crefname{corollary}{Corollary}{Corollaries}
  \Crefname{shadedefn}{Definition}{Definitions}
  \Crefname{section}{Section}{Sections}
  \Crefname{figure}{Figure}{Figures}
  \Crefname{shadeprop}{Proposition}{Propositions}
  \Crefname{exercise}{Exercise}{Exercises}
    \Crefname{assumption}{Assumption}{Assumptions}

\newtheorem{thm}{Theorem}[section]
\newshadetheorem{shadethm}{Theorem}[section]
\definecolor{shadethmcolor}{HTML}{BEFF33}
\definecolor{shaderulecolor}{HTML}{45CFFF}
\newshadetheorem{shadedefn}{Definition}[section]
\definecolor{shadethmcolor}{HTML}{CBFAFA}
\definecolor{shaderulecolor}{HTML}{CBFAFA}
\newshadetheorem{shadeprop}{Proposition}[section]
\definecolor{shadethmcolor}{HTML}{CAFAF2 }
\definecolor{shaderulecolor}{HTML}{CBFAFA}
\newtheorem{conjecture}[thm]{Conjecture}

\newtheorem{lemma}[thm]{Lemma}

\newtheorem{corollary}[thm]{Corollary}
\newtheorem{prop}[thm]{Proposition}

\newtheorem{defn}[thm]{Definition}

\newtheorem{question}[thm]{Question}

\numberwithin{equation}{section}

\theoremstyle{definition}
\newtheorem{remark}[thm]{Remark}

\def \min {\text{min}}
\makeatletter
\def\munderbar#1{\underline{\sbox\tw@{$#1$}\dp\tw@\z@\box\tw@}}
\makeatother

\def\cX{\mathcal{X}}

\def\cV{\mathcal{V}}
\def\cU{\mathcal{U}}

\def\cS{\mathcal{S}}

\def\cO{\mathcal{O}}
\def\cN{\mathcal{N}}
\def\cM{\mathcal{M}}

\def\cK{\mathcal{K}}
\def\cJ{\mathcal{J}}
\def\cI{\mathcal{I}}
\def\cH{\mathcal{H}}
\def\cG{\mathcal{G}}
\def\cF{\mathcal{F}}
\def\cE{\mathcal{E}}

\def\cB{\mathcal{B}}
\def\cA{\mathcal{A}}
\def \Psic {\Psi_{\texttt{contr}}}
\def \Psiuc {\Psi_{\texttt{ucontr}}}
\usepackage{wasysym}

\newcommand{\Pt}{\P^{{\sf t}}}
\newcommand{\Et}{\E^{{\sf t}}}

\def \tauendone {{\sf T}_{\mathsf{stg1}}}

\def\P{\mathbb{P}}
\def\E{\mathbb{E}}

\def\R{\mathbb{R}}
\def\Z{\mathbb{Z}}
\def\N{\mathbb{N}}

\def\R{\mathbb{R}}

\def  \p- {p\textunderscore}

\def\eps{\varepsilon}

\def \tauend {{{\sf T}_{\mathsf{stg}2}}}

\DeclareMathOperator{\Blue}{\mathsf{Blue}}

\DeclareMathOperator{\hittree}{{\sf HITtree}}
\DeclareMathOperator{\hitpast}{{\sf HITpast}}

\DeclareMathOperator{\cont}{\mathsf{Cont}}
\DeclareMathOperator{\tip}{\textsf{tip}}

\DeclareMathOperator{\Tip}{\textsf{TIP}}

\usepackage{extarrows}

\title{The local weak limit of the minimum spanning arborescence in the complete graph}
\author{Swarnadeep Bagchi \thanks{University of Victoria} \and Gourab Ray\thanks{University of Victoria, Research supported partially by NSERC 50311-57400}}
\date{}

\begin{document}
\maketitle
\begin{abstract}
We prove that the local limit of the minimum spanning arborescence in the complete graph (which is an oriented cousin of the minimum spanning tree in the complete graph) is the same as that of the uniform random tree, oriented towards infinity. The latter is known to be the critical Poisson Galton--Watson tree conditioned to survive.  This is in sharp contrast with the local limit of the minimum spanning  tree, which is known to be different from that of the uniform spanning tree (due to results of Addario--Berry \cite{AddarioBerry2013} and Addario--Berry, Griffiths and Kang \cite{PWIT_local_minimum}). Thus we demonstrate that introducing orientations changes the local geometry of the minimum spanning tree in a non-trivial manner.
\end{abstract}
\tableofcontents
\section{Introduction}\label{sec:intro}

The theory of minimum spanning tree (or \textbf{MST}) is now decently well developed with many interesting results being proved  in the past few decades \cite{AddarioBerry2013,Frieze_weight,addario_minimal,PWIT_local_minimum,lyons2006minimal,alexander_minimal}. Recently, the oriented cousin of the minimum spanning tree, called minimum spanning arborescence (or \textbf{MSA} for short), was studied in \cite{RS24}. The preliminary investigation in \cite{RS24} revealed that the MSA has many key features that are distinct from those of the MST. In this article we continue this investigation and study the local structure of the MSA in the complete graph. The local weak limit of the MST has been studied in depth \cite{AddarioBerry2013,PWIT_local_minimum} and one key conclusion in these papers is that the local limit of the MST is different from that of a uniform spanning tree in the complete graph (the latter is alternatively known as the critical Poisson(1) Galton--Watson tree conditioned to survive).
Our main theorem proves that (rather surprisingly!) the local weak limit of the MSA on the complete graph is the same as the critical Poisson(1) Galton--Watson tree conditioned to survive. In particular, the local weak limit of the MSA is \emph{not} the same as that of the MST.

We need a few definitions in order to state our main result. An \textbf{oriented multi-graph} is a pair $(V, \vec E)$ where $V$ is a (finite or countably infinite) set of vertices and $\vec E$ is a set of \textbf{oriented edges}. Informally, each element $\vec e\in\vec E$ is associated with two vertices which we call its \textbf{head} and  \textbf{tail}. We can imagine the oriented edge to be oriented from its tail towards its head. We allow multiple oriented edges to have a common head and a common tail, but we do not allow an oriented edge to have the same head and tail (i.e., we do not allow self-loops).  
Take the (unoriented) complete graph $K_n = ({\sf V}_n, {\sf E}_n)$, with ${\sf V}_n=\{1,2,\ldots,n\}$. Let $\vec {\sf E}_n $ be the set of oriented edges which consists of both orientations of the edges in ${\sf E}_n$. Call $\vec{K_{n}} := ({\sf V}_n, \vec {\sf E}_n)$ the \textbf{fully oriented} complete graph.  Assign i.i.d.\ Exponential$(1)$ weight $W(\vec e)$ to each oriented edge $\vec e\in \vec {\sf E}_n$. Fix a vertex $\partial \in {\sf V}_n$ (and call it the \textbf{boundary}). An \textbf{outgoing edge} from a vertex $v$ is an oriented edge whose tail is $v$. See \Cref{sec:local} for a more formal treatment of these notions.
\begin{defn}
   A \textbf{spanning  arborescence} $T_n$ of $(\vec{K_n}, \partial)$ is a subset of $\vec {\sf E}_n$, where each vertex in ${\sf V}_n \setminus \{\partial\}$ has exactly one outgoing edge, $\partial$ has no outgoing edge, and there are no cycles. 
\end{defn} 
\begin{figure}
    \centering
    \includegraphics[width=0.2\linewidth]{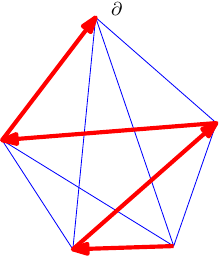}
    \caption{A spanning arborescence of $(K_5, \partial)$ in red.}
    \label{fig:placeholder}
\end{figure}

By absence of cycles, we mean the graph is acyclic even ignoring the orientations (see \Cref{sec:CLEB} for a more detailed treatment).
Another equivalent viewpoint is to treat $T_n$ as a spanning tree of $K_n$ with all its edges `oriented towards' $\partial$. 
The weight of a spanning arborescence is $\sum_{\vec e \in T_n} W(\vec e)$. Let $M_n$ be the (a.s. unique) minimum weight spanning arborescence which we call the minimum spanning arborescence (MSA). 

As mentioned above, we are interested in the local weak limit of the MSA  around a fixed root vertex. Informally, we say two rooted graphs $(G,\rho)$ and $(G',\rho')$ are at distance at most $1/R$ in the local topology, if the graph distance balls around the root of radius $R$ in them match (up to rooted isomorphisms preserving the root and the orientations). We defer to \Cref{sec:local} for more details, but for now we move on and describe the proposed local limit of $M_n$ around the vertex $1$ (all vertices are equivalent for us as we are in the complete graph).
Let $\sf T $ be the critical Poisson Galton--Watson tree conditioned to survive which can be described as follows. Take an infinite path $P$ and call the first vertex $1$. To each vertex attach a critical Galton--Watson tree with offspring distribution Poisson$(1)$. Now let $\vec {\sf T}$ be the tree obtained by orienting `every edge towards $\infty$': orient the path $P$ from $1$ to $\infty$ and orient each tree attached to $P$ towards $P$. It is well-known (see \cite{kesten86} or \cite[Theorem 46]{curiennotes}) that if we take a critical Galton--Watson tree with Poisson$(1)$ offspring distribution and condition it to have $n$ vertices, then the local limit of this object exists and is described by ${\sf T}$. It was proved in the celebrated work of Grimmett \cite{gri_GW} that the local limit of the uniform spanning tree (a spanning tree chosen uniformly at random, no orientations involved) converges to ${\sf T}$ appropriately rooted (see also \cite{nachmias_peres_local} for a more general result). The local limit of the MST, on the other hand, has (arguably) a more complicated description, and, in particular, it is not ${\sf T}$. Our main theorem in this article is:

\begin{figure}
    \centering
    \includegraphics[width=0.6\linewidth]{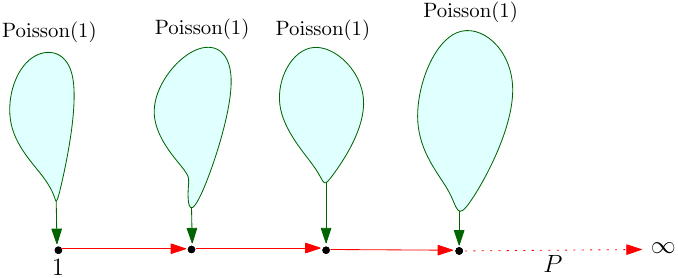}
    \caption{This shows $(\vec {\sf T},1)$ which has the red path $P$ from $1$ to $\infty$ and to all the vertices  Poisson$(1)$ Galton--Watson incoming trees are attached.}
    \label{limit tree}
\end{figure}

\begin{thm}\label{thm:main}
The sequence of rooted oriented graphs $(({\sf V}_n, M_n),1)$ converges weakly to $(\vec{\sf T}, 1)$ in the local topology of rooted marked graphs.
\end{thm}
See \Cref{sec:local} for a more detailed treatment of the local topology in our context of oriented graphs.

One might wonder looking at \Cref{thm:main} whether by some miracle the distribution of the MSA is the same as that of the UST oriented towards $\partial$ (picked with probability proportional to the weights). This is easily seen to be not the case by considering $K_3$ (see \Cref{fig:MSA_triangle}). We also point out here that the main tools used to study the MST, which are Kruskal's or Prim's algorithm \cite{kruskal1956shortest} are not applicable in the oriented setup (see \Cref{fig:kruskal_counter}). In fact the distribution of the MSA is heavily dependent on the distribution of the weights (see \cite[Proposition 3.1]{RS24}).

\begin{figure}[h]
    \centering
    \includegraphics[width=0.8\linewidth]{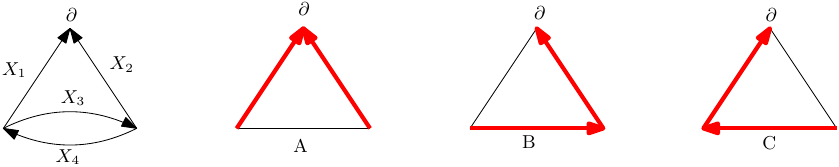}
    \caption{Left: The fully oriented $K_3$ with weights (oriented edges going out from $\partial$ are ignored as they are irrelevant). The arborescence $A$ has probability $1/4$ ($=\P(X_1<X_3,X_2<X_4)$) while each of arborescences $B$ and $C$   has equal probability, which must be $3/8$.}
    \label{fig:MSA_triangle}
\end{figure}

Instead, we rely on an algorithmic procedure due to Chu-Liu, Edmonds and Bock (CLEB) to generate the MSA which was revisited and analyzed in depth in \cite{RS24}. We recall this algorithm in \Cref{sec:CLEB}. For now, it suffices to point out that the CLEB algorithm is much less Markovian compared to Kruskal's or Prim's algorithm. In fact, we rely crucially on the memoryless property of the exponential distribution to extract some Markovian property, and our proof does not work if the weights have some other distribution, for example Uniform$[0,1]$ (see \Cref{sec:open} for some open problems in this direction).

\begin{figure}
    \centering
    \includegraphics[width=0.5\linewidth]{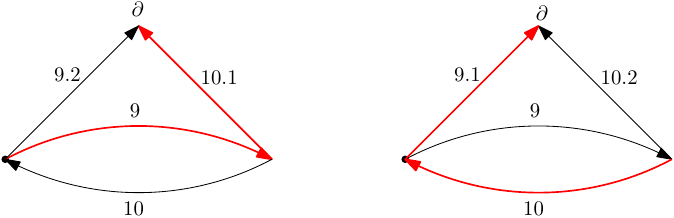}
    \caption{An example showing that any algorithm which is only a function of the ordering of the edges according to their weights cannot output the MSA (the red edges depict the MSA). Consequently, any reasonable variation of Kruskal's algorithm cannot output the MSA.}
    \label{fig:kruskal_counter}
\end{figure}

\subsection{Two different heuristics behind \Cref{thm:main}.}
In this section we give some heuristic justification of why \Cref{thm:main} holds. Let $U_n$ be the set of minimum weight outgoing edges from all the vertices in $\vec K_n $ except $\partial.$ It is not too hard to see that this breaks up $\vec K_n$ into several components, and each component not containing $\partial$ is an \emph{oriented cycle rooted tree}. More precisely, each component has a unique oriented cycle with trees attached to the vertices of the cycle, all oriented towards the cycle. The component containing $\partial$ is a tree oriented towards $\partial$. It is easy to see that the number of oriented edges incoming into a vertex is distributed as Binomial$(n-1,(n-1)^{-1}) \approx $ Poisson$(1)$ since the weights are i.i.d.\ continuously distributed. Some more thought in this direction allows us to conclude that the local structure of $U_n$ is indeed the same as $\vec {\sf T}$. Therefore, another way to view \Cref{thm:main} is that the local structure of the MSA is `undisturbed' if one `repairs' $U_n$ to obtain $M_n$. It is not at all obvious that this is the case, as there can be potentially many arborescences with weight close to the minima, which look completely different from the MSA. In fact the CLEB algorithm (see \Cref{sec:CLEB}) exactly tells us a technique to do this repair, perhaps somewhat implicitly, and the whole argument in the article is to establish such local stability. See \cite{ChatterjeeRay2024} for some recent results on stability of optimization problems under perturbations. 

Now we present another perspective behind the local structure of the MSA by considering the MSA as a subgraph of the ambient graph $\vec K_n$. Observe that  scaling the i.i.d.\ Exponential$(1)$ weights by a positive constant does not change the distribution of $M_n$ (although it does affect the weight of $M_n$). In this section, suppose the weights are i.i.d.\ Exponential $((n-1)^{-1})$ (which has expectation $n-1$ and in turn the minimum weight outgoing edge has order 1 weight). After this scaling the order statistic of the weights of the outgoing edges from any vertex in $\vec K_n$ converges to the arrival times of a Poisson process with rate $1$ as $n \to \infty$. Also the same holds for the set of incoming edges to any vertex.

We now describe an object which is the local limit of the weighted fully oriented complete graph where the weights are as above. We call it the oriented Poisson weighted infinite tree (OPWIT) which is the so-called oriented version of the Poisson weighted infinite tree or PWIT (see e.g. \cite[Section 4.2]{AS_90}). We describe the OPWIT recursively as follows. We start with a root vertex $\rho$. Given any oriented edge $\vec e$ going from $u $ to $v$, we define $v$ to be the \textbf{head} of  $\vec e$ (denoted $\vec e_+$) and $u $ to be the \textbf{tail} of $\vec e$ (denoted $\vec e_-$). In the OPWIT, 
the set of outgoing edges from $\rho$ is countably infinite. There is also a countably infinite set of oriented edges incoming to it. Call the heads of these outgoing edges and the tails of these incoming edges the children of $\rho$. The weights on the outgoing edges, when ordered in increasing order, are given by the arrival times of a Poisson process with rate 1. The same is true for the incoming edges (for an independent Poisson process). Now recursively, having defined a set $\cS$ of oriented edges and their endpoints, choose any vertex $v$ whose children are not defined. The vertex $v$ has infinitely many outgoing edges, none of whose heads are in $\cS$. The vertex $v$ also has infinitely many incoming edges one of which is from its parent, and none of the tails of the other incoming edges are in $\cS$. The heads of these outgoing edges from $v$ and the tails of the incoming edges to $v$ (except the one from its parent) are the children of $v$.
 The ordered weights of the outgoing edges from $v$ and the ordered weights of the incoming edges to $v$ (except the one coming from the parent)  are again the arrival times of two i.i.d.\ Poisson processes with rate $1$, independent of everything else. This is continued \emph{ad infinitum}.  Note that this describes naturally a set of vertices in the $n$th generation for $n \in \N$, with $\rho $ being the common ancestor of every vertex.

Let $W_n=(W(\vec e))_{\vec e\in{\vec {\sf E}_n}}$.
It can be shown that the weighted rooted graph $(\vec K_n, W_n, \rho)$ converges locally to the OPWIT rooted at $\rho$. The topology used for this convergence can be made precise, see e.g. \cite[Section 2.2]{AS_90}. Roughly, the weight on an edge can be thought of as the length of the edge and the orientation as a `mark' on the edge. This defines a natural metric on the graph which can be used to define a local topology on rooted, oriented, weighted graph. We skip these details here and refer to \cite{aldous_assignment} for a proof of the same theorem in the unoriented setup. We will not prove this result here as it is almost identical to the proof of the unoriented case, and is somewhat tangential to the main agenda of the article.

Now we want to point out that the component of $\rho $ of the MSA on OPWIT is the same as that described in \Cref{thm:main}. One way to see this precisely is to consider the so-called wired limit of the OPWIT:  glue all the vertices in the $n$th generation of OPWIT into a single vertex $\partial$, call it OPWIT$_n$. The minimum spanning arborescence $\tilde M_n$  of OPWIT$_n$ with boundary $\partial$ (although it still has infinitely many vertices) is well defined and is given  simply by taking the union of the minimum weight outgoing edge from every vertex and considering the component of $\rho$. This is easy to see as there are no oriented cycles in OPWIT$_n$. Thus the component of $\rho$ has a path going from $\rho $ to $\partial$ following the minimum weight outgoing edges. It is not too hard to see that the set of incoming edges to a vertex which belongs to $\tilde M_n$ is given by an inhomogeneous Poisson process $N(\lambda(t))$  with rate $\lambda (t)= e^{-t}$. Thus the total number of incoming  oriented edges into any vertex is Poisson$(1)$ independent of everything else. Using the memoryless property of Exponential, this argument can be extended to show that the incoming tree has distribution given by a Poisson $(1)$ Galton--Watson tree. 

This gives a heuristic justification of the appearance of the limit in \Cref{thm:main}. Of course, we need to show that the order of the limit of the ambient weighted graph and the MSA can be exchanged, which is the primary goal of the rest of the paper. 

\subsection{Outline of the proof}\label{sec:outline}
The key tool is the CLEB algorithm which was developed back in the 70s by computer scientists \cite{Chu-Liu,Edmonds,bock,Karp} in order to implement a fast algorithm to sample the MSA. This was revisited by Ray and Sen in \cite{RS24} with the goal of studying the geometry of MSA. The description of the algorithm is somewhat involved, so we attempt to give a rough description of the algorithm in the special case of the Exponential distribution, hiding some of the details. See \Cref{sec:CLEB} for more details.
\medskip

\noindent
\textbf{A primer on the CLEB algorithm.}
\begin{figure}
    \centering
    \includegraphics[width=0.5\linewidth]{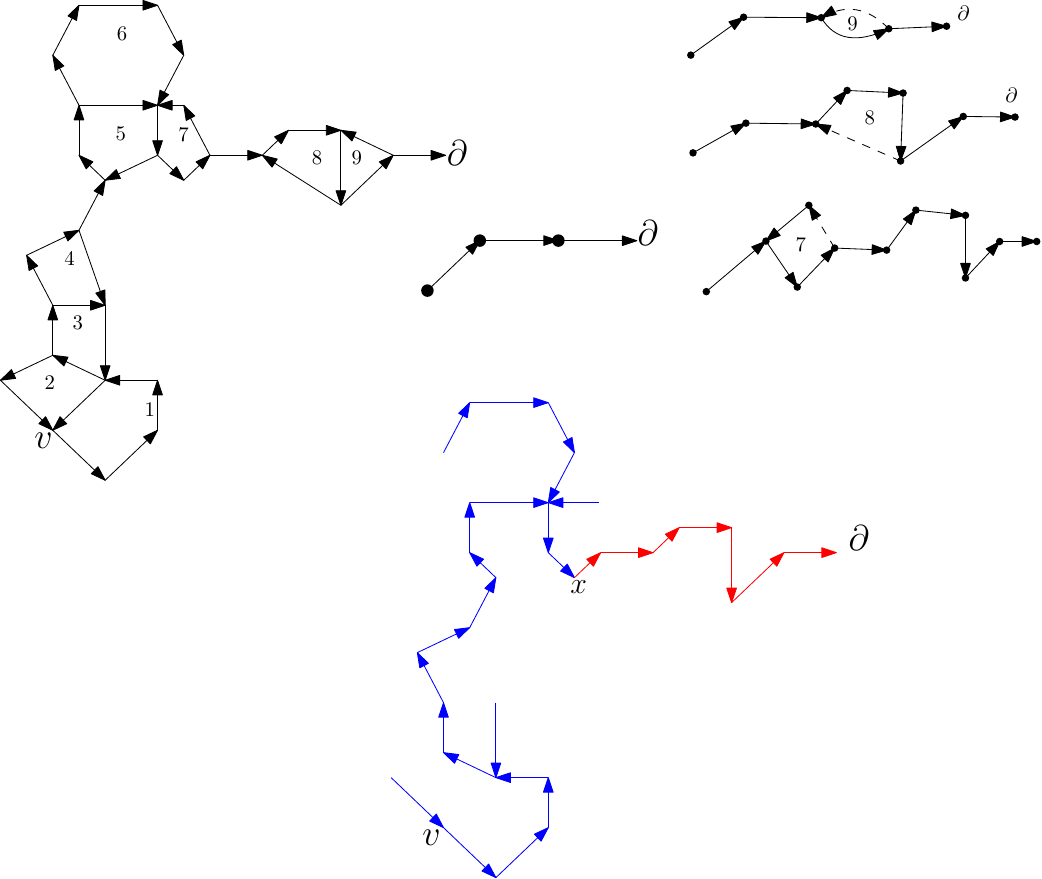}
    \caption{An illustration of the CLEB algorithm. Top left: The cycles contracted in order are numbered. Top middle: the arborescence obtained after fully contracting every cycle. Top right: the first three steps of the uncontraction process. Bottom: the arborescence obtained after fully running the uncontraction step. The future of $x$ is shown in red and the past of $x$ is shown in blue.} 
    \label{fig:CLEB}
\end{figure}
We encourage the reader to refer to \Cref{fig:CLEB} while reading this primer.
Suppose we are given a weighted fully oriented graph $G$.
The CLEB algorithm proceeds in two phases, called the \emph{contraction phase} and the \emph{uncontraction phase.} The contraction phase is described as follows. Pick a vertex $v \neq \partial$ and expose the minimum weight outgoing edge $\vec e$ and subtract the weight of $\vec e$ from the weight of all the other outgoing edges from $v$ (this weight subtraction is a \emph{gauge transform} which keeps the MSA unchanged).  Declare $v$ to be \emph{explored}.   Then pick another unexplored vertex except $\partial$ (arbitrarily) and expose the minimum weight outgoing edge and subtract weights as before. Keep doing this until a cycle is created. Then \emph{contract} the cycle (glue all the vertices of the cycle together and remove the self-loops thus created). Declare this contracted vertex unexplored. Continue until there is no unexplored vertex left except $\partial$. This marks the end of the contraction phase.

 Now let us describe the uncontraction phase (see \Cref{fig:uncontraction}). Order the cycles contracted chronologically. Now `open' the last contracted cycle. This creates an object which is a spanning  arborescence except one of its vertices has been inflated into a cycle. This object has exactly one vertex with two outgoing edges. Remove the edge belonging to the cycle. Iterate this process: open the cycles in reverse chronological order and remove edges as described. In the end, we obtain an arborescence $T$ of the original graph $G$. The CLEB algorithm asserts that $T$ is the minimum spanning arborescence. 
 
\medskip

\noindent
\paragraph{The strategy of the proof.}
Let us first define a generalization of the spanning arborescence of a graph with a boundary which we simply call arborescence.
\begin{defn}\label{def:general_arborescence}
Given an oriented graph $G = (V, \vec E)$, an \textbf{arborescence} is a subset of oriented edges such that every vertex has either a single outgoing edge or no outgoing edge, and there are no cycles. The set of vertices which have no outgoing edge is called the boundary of the arborescence. 
\end{defn}
 So in other words, if the boundary of an arborescence is a single vertex $\partial$ then this arborescence is a spanning arborescence of $(G, \partial)$. Now simply observe that every step of the CLEB algorithm produces an arborescence in a graph formed by possibly contracting some cycles in the original graph $G$. Furthermore, the boundary of this arborescence is simply the set of unexplored vertices. If there are no more vertices left to explore except $\partial$ in step $i$, we obtain a spanning arborescence of the graph in step $i$ with boundary $\partial$.

Observe that the CLEB algorithm allows an arbitrary order of the choices of $v$ to explore. This along with the memoryless property of exponential random variables simplifies the law of the CLEB process considerably: namely, conditioned on the oriented edges exposed and the weights, the choice of the minimum weight outgoing edge from an unexplored vertex $v$ is uniform among the outgoing edges in the contracted graph in that step. Now consider the following choice of ordering of vertices to explore. Arbitrarily choose the first vertex. Suppose in a generic step we expose an edge $\vec e$. If a cycle $C$ is created which is contracted to $v_C$, choose $v_C$ in the next step. Otherwise, choose $\vec e_+$. Stop if $\vec e_+ = \partial$. This creates the following rather elegant stochastic process: perform a simple random walk and as soon as a cycle is created, contract it. Then continue doing simple random walk started from the contracted vertex in the contracted graph, and stop if $\partial$ is hit. We refer to \href{https://graymath24.github.io/webpage_G/pictures.html}{this URL} for a simulation of this process in the square lattice. We call this process the CLEB walk. (This process was also called the loop contracting random walk in \cite{RS24}, but we stick to CLEB walk in this article). 

This gives a much simpler method to analyze the CLEB process. Pick a vertex and perform CLEB walk until $\partial $ is hit. Suppose $W$ is the set of vertices exposed (which includes $\partial$). Now pick another arbitrary unexplored vertex and perform CLEB walk until $W$ is hit. Now replace $W$ by the union of $W$ and the set of new vertices thus exposed. Continue until there is no vertex left to explore except $\partial$. Then we uncontract the cycles in the uncontraction phase to obtain the MSA. Thus if we have good control on the CLEB walk, particularly of the cycles contracted involving the vertices near the root, we are in good shape to prove a local limit theorem. As it turns out, in our setup, too many cycles are contracted in the first walk (which creates $W$) itself. We will now attempt to explain this.
\begin{figure}[h]
    \centering
    \includegraphics[width=0.8\linewidth]{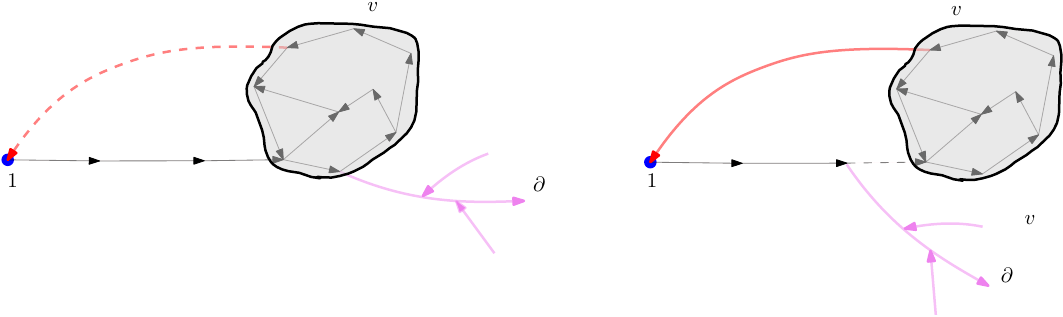}
    \caption{The giant is marked in gray. Left: If we hit 1 and escape from the giant, then the edge hitting 1 is erased in the uncontraction process. Right: If we hit 1 and then we do not escape from the giant, then the giant and the red edge is added to 1.}
    \label{fig:good_bad}
\end{figure}

The CLEB walk is particularly nice to analyze if the cycles created are of somewhat small sizes.
However, this is unfortunately not the case in the complete graph, let us argue why. If we perform the CLEB walk started from $1$, then in each step, the probability to hit back $1$ is roughly $O(1/n)$. Consequently it would take roughly $O(n) $ steps to see a cycle involving $1$ (or exposed vertices near $1$). Also, for the same reason, the time taken to hit $\partial $ is also $O(n)$.

Observe that a vertex $v$ of the contracted graph in the $i$th step contains a subset of the vertices of the original graph, which were contracted to create $v$. We call this set of vertices, $\cont(v)$ (a formal definition is in \Cref{sec:2 and 3}). Now notice that because of the birthday paradox, cycles start getting created by time $\sqrt{n}$. In fact, by time $O(\sqrt{n})$, a vertex $v$ will be created with $\cont(v)$ being at least $C\sqrt{n}$ with very high probability. Let's call this vertex the `giant' in this step.  From this point on, we can break up the CLEB walk into excursions, which start from the giant and end by either hitting the giant or `swallowing' it by hitting a vertex in its past. Since we are on the complete graph, if the giant size at the end of one excursion is $g$, the next excursion length is roughly Geometric with parameter $g/n$.
This creates a self-reinforcing mechanism which grows the giant and the excursions get smaller and smaller as the giant grows. 
  Thus at any time $t \gg \sqrt{n}$, the giant size is also of order $t$ and there are still $O(\sqrt{n})$ vertices which are not in the giant (in fact it is much smaller). We call this the `snowball effect'.
  
Now by time $n$, $1$ has a positive chance of being hit before $\partial$ is hit. Thus in order to understand the geometry of the MSA near $1$, we need to untangle the contracted vertices that have snowballed by time $O(n)$ into a giant, contracted vertex of size $O(n)$.
\begin{figure}
    \centering
    \includegraphics[width=0.5\linewidth]{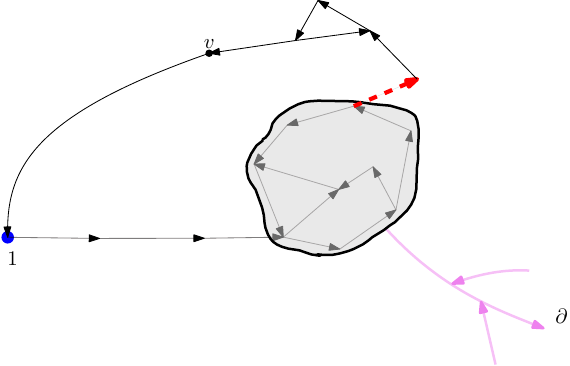}
    \caption{A bad event which can happen with good probability in the naive approach}
    \label{fig:bad_naive}
\end{figure}

Now we observe that there is a certain good event which, if it happens, ensures that the neighbourhood of 1 is `undisturbed'. We now describe this event.
Suppose in a certain step we explore $v$ and we hit $1$ by exposing an oriented edge $\vec e$. This creates a cycle, call it $C$. If in the next step we `escape' using an oriented edge outgoing from $v$, then the neighborhood of 1 remains undisturbed. Indeed, if this happens, in the uncontraction phase when we uncontract $C$, we necessarily erase $\vec e$ and the neighborhood of $1$ remains undisturbed (see \Cref{fig:good_bad}). It is now clear that what matters is the vertex from which the last excursion starts is a good vertex like in the left of \Cref{fig:good_bad}. The actual definition of these good vertices
is somewhat more involved and is described in detail in \Cref{sec:2 and 3} (we call the bad vertices blue there). But for the moment, this describes the essence of the type of events we need to consider.

But what is the probability that such a good event happens? If $v$ is actually the giant, then $v$ takes up most of the mass of the set of exposed vertices, and hence with high probability the good event happens. However, because of the same reason, the good event cannot happen if $v$ is not the giant: in this case escaping from $v$ is minuscule as the process escapes from the giant with overwhelmingly high probability (see \Cref{fig:bad_naive}). Thus if we choose this route, we need to track the cycle lengths whenever a neighborhood of 1 is hit from a non-giant vertex. The number of times the explored vertex is non-giant is of order $n$ and hence there is a positive chance that the neighborhood of 1 is hit order 1 times from a non-giant vertex. Uncontracting and tracking the cycles thus created is somewhat complicated and perhaps intractable.

In order to track the paths which `attach' to 1 through events depicted in \Cref{fig:bad_naive}, we use the flexibility of the CLEB algorithm even further as follows. We divide the algorithm into three stages. In \textbf{stage 1}
we expose the minimum weight outgoing edges  from all the vertices except $\partial$, call them $\cE$. Observe that this naturally defines an oriented path started from any vertex $v$ by first considering the edge $\vec e \in \cE$ whose tail is $v$ and then considering the edge ${\sf f} \in \cE$ whose tail is $\vec e_+$ and so on. Iterating this we obtain a path which either ends in $\partial$ or ends up creating a cycle. Call this path the `future' of $v$.
Now retain only those vertices, whose future hits $1$. This is the `minimum weight incoming tree' into $1$, which is indeed an arborescence with boundary $1$ (see \Cref{def:general_arborescence}). We call this arborescence ${\sf t}$. It is not hard to see that the distribution of the number of incoming edges to any vertex is roughly Binomial $(n-1,(n-1)^{-1}) \approx $Poisson$(1)$. This essentially shows that the incoming tree is distributed as a Poisson(1) Galton--Watson tree, which ensures that the local structure matches with the limiting arborescence ${\sf T}$, at least up to this point. Let us mention one small detail here: we don't really expose all of $\cE$, but only the incoming tree using a standard depth first search exploration procedure, so the remaining edges can be independently explored using CLEB. In \textbf{stage 2} we perform CLEB walk from $1$ until $\partial$ is hit. Then in \textbf{stage 3} we do CLEB walk from the remaining vertices in arbitrary order.

Now observe that once this incoming tree ${\sf t}$ is exposed, the law of the minimum weight outgoing edges from the unexplored vertices is distributed uniformly  among those edges \emph{which do not   hit ${\sf t}$}. So now, if an explored vertex $v$ is a singleton (it is a vertex of the original graph $G$ which is being explored),  it cannot hit ${\sf t}$. So the only way a  non-giant explored vertex hits ${\sf t}$ is if it is not a singleton (i.e. it is a contracted vertex, but not the giant). Next we argue that the set of times  when the vertex $v$ explored has $\cont(v)$ strictly bigger than 1, and is not a giant is $o(n)$. Indeed, because of the snowball effect argued above, a small cycle has to be created in the middle of the excursions between the times when the giant is hit or swallowed. This has a probability which is small enough so that a union bound allows us to conclude. 

Finally, in stage $3$, we see that we have already exposed $O(n)$ many vertices. So each of the CLEB walks is of size $O(1)$ with high probability, since in every step, the walk has a positive chance to hit an exposed vertex. Again, since ${\sf t} $ cannot be hit from a singleton, the only way to hit ${\sf t}$ is from a contracted vertex. The set of times when this is the case is $o(n)$ in expectation. Thus a union bound shows that ${\sf t}$ is not even hit in stage 3 with very high probability.

In order to make the above strategy work, we need some quantitative estimates of the sizes of the giant,  the growth of the bad vertices inside the giant (which we color blue later), which needs more in depth analysis of the CLEB walk in stage 2.

\subsection{Local topology}\label{sec:local}
In this section, we describe the metric space of \emph{rooted marked graphs} and we will describe the topology for convergence of the MSA in the language of the local convergence for rooted marked graphs. We refer to \cite{AL_unimodular} for a detailed exposition of this topic, and recall here just the essentials.

A \textbf{multi-graph} $G=(V,E)$ consists of a set of vertices $V$ and edges $E$. Any two distinct vertices can be joined by a finite number of edges; the collection of all such edges is denoted by $E$. We emphasize that there is no self-loop in a multi-graph. We will only consider the multi-graphs where $V$ and $E$ are at most countable. For $u,v\in V$, we denote the set of edges between $u$ and $v$ by $E_{uv}$. A \textbf{rooted multi-graph} is a pair $(G,o)$, where $G=(V,E)$ is a multi-graph and $o\in V$ is a fixed vertex. Furthermore, a rooted or non-rooted multi-graph is called \textbf{locally finite} when each of its vertices has finite degree. In this article, we will only work with locally finite multi-graphs. A \textbf{rooted marked graph} $(G,o,\mathfrak m)$ is a rooted multi-graph $(G,o)$ with a function $\mathfrak m:E\to\cM$, where $\cM$ is a complete and separable metric space often called the space of \textbf{marks}.

We can consider a \emph{rooted oriented multi-graph} $(G,o)$ as a rooted marked graph $(G,o,\mathfrak m)$, where the space of marks is $V\times V$ with discrete metric, such that $\mathfrak m(e)=(u,v)$ or $(v,u)$ if $e\in E_{uv}$. Let $\vec E:=(e,\mathfrak m(e))_{e\in E}$ and an element of $\vec E$ is denoted by $\vec e$. The set $\vec E$ is called the set of \emph{oriented edges}. If $\mathfrak m(e)=(u,v)$, we call $v$ the \textbf{head} of $\vec e$ and $u$ is the \textbf{tail} of $\vec e$ and denote them as $\vec e_{+}=v,\vec e_{-}=u$. In this article, we will often consider \textbf{fully oriented} graphs which can be defined as follows: take a multi-graph $G$ and orient each edge in both directions, the new oriented multi-graph which we obtain, is called fully oriented graph (sometimes called fully oriented version of $G$). We also say two rooted oriented multi-graphs are isomorphic if there is a graph isomorphism between them which preserves the marks and the root. We let $\cG^{\bullet}$ be the space of equivalence classes of rooted oriented multi-graphs.

Let $(G,o,\mathfrak m)$ be a marked graph. The graph distance in $G$ is the usual graph distance on the unoriented graph $G$, which is simply the length of the shortest path connecting two vertices.
We will use the notation $B_{(G,o,\mathfrak m)}(r)$ to denote the ball of radius $r$ around $o$ in $G$. A random rooted marked graph is simply a random variable taking values in $\cG^{\bullet}$.

\begin{defn}\label{defn:loc_conv}
A sequence of random rooted marked graphs $\{(G_n,o_n,\mathfrak m_n)\}_{n\ge 0}$ converges locally weakly to $(G,o,\mathfrak m)$ if for all $(H,u,\mathfrak q)\in\cG^{\bullet}$ and $r\ge 0$,
$$\P(B_{(G_n,o_n,\mathfrak m_n)}(r)\simeq (H,u,\mathfrak q))\to\P(B_{(G,o,\mathfrak m)}(r)\simeq (H,u,\mathfrak q)),$$
as $n\to\infty$, where $a \simeq b$ means $a$ is isomorphic to $b$ as rooted marked graphs.
\end{defn}
\begin{remark}
It can be shown that this definition of local weak convergence is equivalent to the usual definition of weak convergence in the local topology. The local topology is induced by the metric $d_{\text{loc}}$ defined as follows. We define $d_{\text{loc}}((g,o,\mathfrak m), (g',o',\mathfrak m')) =\frac1{R+1}$ where $$R = \max\{r \ge 0: B_{(G,o,\mathfrak m)}(r)\simeq B_{(G',o',\mathfrak m')} \}.$$ 
\end{remark}
Observe that the sequence of MSAs of fully oriented complete graph with i.i.d. Exponential$(1)$ edge weights lie in $\cG^{\bullet}$, also $\vec{\sf T}$ we have defined in \cref{sec:intro} lies in $\cG^{\bullet}$. Therefore we need to verify \Cref{defn:loc_conv} to prove \cref{thm:main}.
\section{The Chu-Liu Edmonds Bock (CLEB) algorithm}\label{sec:CLEB}
The key tool in the analysis is an algorithm to sample the MSA, which works in a general graph. This algorithm is attributed independently to Chu and Liu \cite{Chu-Liu}, Edmonds \cite{Edmonds} and Bock \cite{bock} (see also Karp \cite{Karp}). We describe a version of the algorithm which is relevant for this article, we refer to \cite{RS24} for relevant background and exposition.

In what follows, we work with an oriented multi-graph (see \Cref{sec:local} for a formal definition). We denote the head and tail of an oriented edge $\vec e$ by $\vec e_{+}$ and $\vec e_{-}$, respectively. An \textbf{oriented path} is a sequence of distinct oriented edges $(\vec e_0, \vec e_1, \ldots, \vec e_{k})$ such that the head of $\vec e_j$ is the tail of $\vec e_{j+1}$ for $0 \le j \le k-1 $, and such an oriented path is a \textbf{cycle} \footnote{In the literature on {oriented} graphs, an oriented path and an oriented cycle are sometimes referred to as a trail and a circuit, respectively.} if the head of $\vec e_k$ is the tail of $\vec e_0$.  The endpoints of these edges are called the vertices of the path or cycle.

Before describing the CLEB algorithm, let us define two graph operations that will be used throughout the paper (see \Cref{fig:uncontraction}).

\begin{itemize}

\item {\textbf{ Contraction.}} Given an oriented multi-graph $G = (V, \vec E)$ with boundary {$\partial \subset V$} and a cycle $C$ in $G$ not passing through $\partial$, this operation {\bf contracts  the cycle $C$} and outputs another {oriented} multi-graph $G' = (V', \vec E')$ with boundary $ \partial$, which we will denote by
\[ \Psic ((G, \partial), C) = (G', \partial). \]
Informally, the operation is simply `gluing' all the vertices in $C$ into a single vertex, and removing all the self-loops thus created. Below we give a formal description of this operation.

To obtain $G'$ from $G$, we remove all vertices in $C$ and add a new vertex $v_C$ (which we can think of as the ``contracted vertex" formed by identifying all vertices in $C$ with $v_C$) while keeping the rest of the vertices unchanged (including the boundary $\partial$). In $G'$, we retain all the oriented edges of $G$ both of whose endpoints are outside $C$. We then remove all the oriented edges with both endpoints in $C$. Finally, for every oriented edge in $G$ with exactly one endpoint in $C$, we replace it with the following oriented edge in $G'$.  For every edge  in $G$ with head (resp.\ tail) $u \not \in C $ and tail (resp.\ head) in $C$,  replace it by an edge $\vec e$ in $G'$  with $\vec e_+ = u$ (resp.\ $\vec e_- = u$) and $\vec e_- = v_C$ (resp.\ $\vec e_+ = v_C$).

We remark that the removal of all the edges with both endpoints in $C$ ensures that there is no self-loop in $G'$.
Let us also remark that there is a natural bijection between $\vec E'$  and the edges in $\vec E$ with at most one endpoint in $C$, which yields an embedding of $\vec E'$ inside $\vec E$.   Hence, we can (and will) think of the elements $\vec E'$ as elements of $\vec E$ as well. However, the tail and head of an oriented edge in $\vec E'$  may be different from those of its image in $\vec E$. Also, given any collection of edges $\vec F \subseteq \vec E$, each edge from $\vec F$, both of whose endpoints do not belong to the vertices of $C$, is uniquely mapped to an edge in $\vec E'$. We will denote the resulting collection of edges in $\vec E'$  obtained from this mapping by $\vec F \cap \vec E'$. 

\begin{figure}[h]
    \centering
    \includegraphics[scale = 0.7]{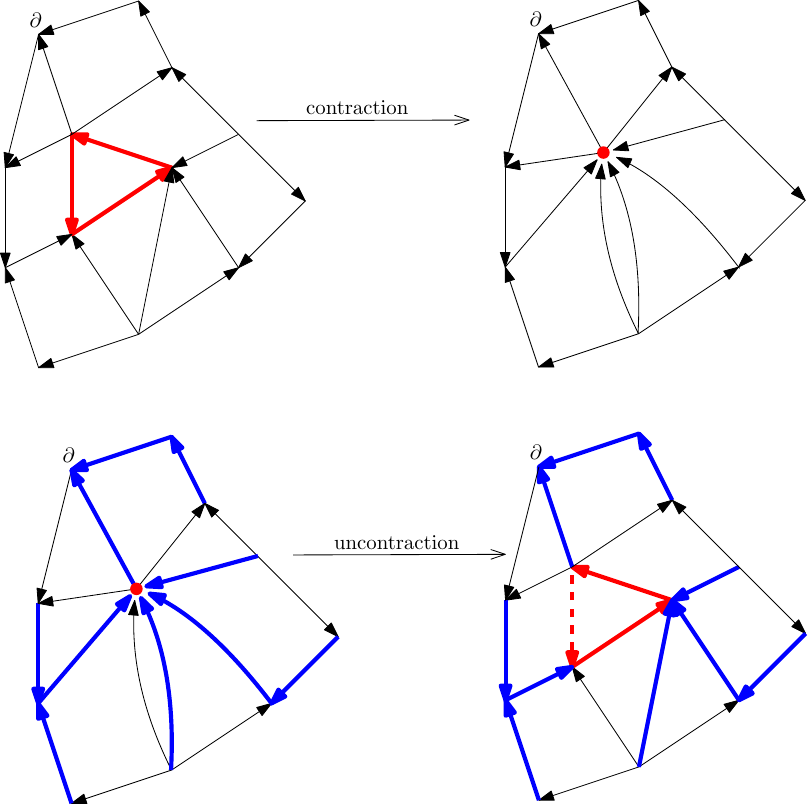}
    \caption{The contraction and uncontraction procedure. Top: An oriented multi-graph $G$ with the cycle $C$ is marked in red which gets contracted.
    Bottom: $(G', \partial) = \Psic ((G, \partial), C)$ with a spanning arborescence $\vec T'$ in blue. Right: The graph $G$ with spanning arborescence $\vec T:=\Psiuc((G, \partial), (G', \partial), C,\vec T')$ which is the union of the blue and the red oriented edges. The edge in $C$ which is deleted is dotted.}
    \label{fig:uncontraction}
\end{figure}

\item {\textbf{ Uncontraction.}} Suppose we have contracted a cycle $C$ in $(G, \partial)$ to obtain $(G', \partial)$, i.e., 
$\Psic ((G, \partial), C) = (G', \partial)$.   Let $\vec T'$ be a spanning arborescence of $(G', \partial)$. Given this data, the uncontraction operation yields a spanning arborescence $\vec T$ in $(G, \partial)$ by  {\bf uncontracting  the cycle $C$}, which we will denote by 
\[ \Psiuc ((G, \partial), (G', \partial), C, \vec T') = \vec T. \]

Let $\vec T^\star \subseteq \vec E$ be the union of the edges in $C$ and the image of the edges of $\vec T'$ in $G$. Note that $\vec T^\star$   almost forms a spanning arborescence in $(G, \partial)$ except for one blemish. There is a unique vertex of $C$ which has exactly two outgoing edges in $\vec T^\star$, one in $C$, say $\vec e_C$, and the other in the image of the edges in $\vec T'$, say $\vec e_{T'}$. Define $\vec T =\vec T^\star \setminus \{\vec e_C\}$. It is easy to see that $\vec T$ is a spanning arborescence of $(G, \partial)$.

\end{itemize}
\subsection{The sequential CLEB algorithm}
For any countable set $N$, we say a collection of weights $(w_x)_{x \in N}$ is \textbf{generic} if 
\begin{equation}
\sum_{x \in S}  n_{x} w_{x} \neq 0, \text{ for all finite $S \subseteq N$,} \label{eq:linear_comb}
\end{equation}
and any choice of coefficients $n_{x} \in \Z$, $x \in S$ such that $n_x \neq 0$ for some $x \in S$.
If the weight collection is i.i.d.\  with a continuous distribution, then clearly it is a.s.\ generic. 
It is also clear that a generic collection of weights on $\vec E$ produces a unique MSA for $(G, \partial)$.

For a vertex $v$ in a graph $G = (V, \vec E)$, let $\cO(v)$ denote the set of outgoing edges from $v$ in $G$, i.e., $$ \cO(v) = \{\vec e \in \vec E: \vec e_- = v\}.$$
Given a collection of generic weights $w:=(w_{\vec e})_{\vec e \in \vec E}$,
let $e_{\min}(v,w)$ be the oriented edge with minimum weight $w$-weight among $\cO(v)$. If the weights $w$ are clear from context, we sometimes drop it from the notation and write it as $e_{\min}(v)$.

 Take a finite, connected fully oriented graph $G = (V, \vec E)$ and a boundary vertex $\partial$ and a generic collection of weights $w$. At any point of time in the algorithm, there is a set of `exposed' edges $S_j$, an oriented multi-graph $(V_j, \vec E_j)$ (which is formed by possibly contracting cycles in $G$), and a collection of generic weights $w_j:=(w_{j,\vec e})_{\vec e \in \vec E_j}$. The vertex $\partial $ is in $V_j$ for every $j \ge 0$.  Informally, the sequential CLEB algorithm can be described by the following steps.
\begin{itemize}
    \item Pick a vertex $v \neq \partial$ in $V_j$  and  such that $e_{\min}(v, w_j)$ is not exposed and suppose its weight is $\pi_v>0$.  Expose this edge and subtract $\pi_v$ from the $w_j$-weights of all the outgoing edges from $v$. 
    \item If a cycle is created, contract it.

\item If there is no  vertex for which we can execute the first item, stop the process. Uncontract the edges in reverse chronological order. 
\item The subset of edges obtained in this way is the minimum spanning arborescence of $(G, \partial)$ for weights $w$.
\end{itemize}
Below we provide a formal description of the above process.
\paragraph{The sequential CLEB algorithm.}  {\em Forward direction.} Let us inductively define a sequence $$H_j := (V_j, \vec E_j, S_j, (w_{j,\vec e})_{\vec e \in \vec E_j})_{j \ge 0}$$ 
where $G_j:= (V_j, \vec E_j)$ is an { oriented} multi-graph with weights $(w_{j,\vec e})_{\vec e \in \vec E_j}$ and $S_j$ is a subset of $ \vec E$.  We shall call $S_j$ the set of \textbf{exposed edges} by the algorithm at step $j$. A vertex $v$ is an \textbf{exposed vertex} at step $j$ if at least one of the outgoing edges from $v$ is in $S_j$. For $j = 0$, we take $G_0 = G$, $S_0 = \emptyset$ and $w_{0,\vec e}$ = $w_{\vec e}$ for $\vec e \in \vec E_0 = \vec E$. 

The process will be such that $G_{j}$ either is the same as $G_{j-1}$ or is obtained by contracting a single cycle of $G_{j-1}$. Therefore, we can view $\vec E_j \subseteq \vec E_{j-1} \subseteq \vec E$ and $S_j \cap \vec E_j$ as a collection of edges in $\vec E_j$.  
At every step $j$, $S_j \cap \vec E_j$ is an arborescence of $G_j$ with boundary given by exactly those vertices whose tail is not in $S_j \cap \vec E_j$ (recall \Cref{def:general_arborescence}). This is clearly true for $j=0$ where the arborescence is the empty set with boundary $V$.   To carry out the inductive definition, we assume that we have generated $H_0, H_1, \ldots, H_{j-1}$ for $j \ge 1$. 

\paragraph{Contraction phase.}
\begin{enumerate}[(a)]
\item {\textsf {(Choosing a vertex.)}} In step $j$, pick a {vertex $v \ne \partial$} in $V_{j-1}$ which has no exposed outgoing edge at step $j-1$. The choice of this vertex may be a deterministic function of $(H_{k})_{k \le j-1}$, the weights $ w$ and possibly some additional independent source of randomness. If there is no such vertex $v$, stop the  contraction phase of the algorithm and go to the uncontraction phase.

\item {\textsf {(Reveal $\vec e_{\min}(v)$.)}}  Let $\pi_v$ be the minimum weight among the outgoing edges of $v$. Let $$w'_{j-1,\vec f} = w_{j-1,\vec f} - \pi_v, \text{ for all } \vec f \in \cO(v)$$ and $w'_{j-1,\vec f} = w_{j-1,\vec f}$, for all other $\vec f \in \vec E_{j-1}$. Define $S_j =  S_{j-1} \cup \{\vec e_{\min}(v)\}$. 

\item {\textsf {(If no cycle is created, move on)}} If $S_j  \cap \vec E_{j-1}$ has no cycle (that is, a cycle is not created by exposing $\vec e_{\min}(v)$), declare $V_{j} = V_{j-1}$, $\vec E_{j} = \vec  E_{j-1} $ and $w_{j,\vec e} = w'_{j-1,\vec e}$, for all $\vec e \in \vec E_j$.  

\item \textsf {(If a cycle is created, contract)} If $S_j  \cap \vec E_{j-1}$ has a cycle {$C=C_{j-1}$} (that is, a cycle is created by exposing $\vec e_{\min}(v)$), contract the cycle $C$ in $G_{j-1}$  to obtain $G_j = (V_j, \vec E_j)$, that is, $$(G_{j}, \partial) := \Psic((G_{j-1}, \partial), C ).$$
Then define the new weights as $w_{j,\vec e} = w'_{j-1,\vec e}$ for all $\vec e \in \vec E_j$. 
\end{enumerate}
Recall  the definition of an arborescence from \Cref{def:general_arborescence}, in particular recall that we allow multiple vertices to be the boundary of an arborescence. Given an arborescence $\vec t$ and a vertex $v$ with one of the outgoing edges from $v$ in ${\sf t}$, there is a unique oriented path in $\vec t$ from $v$ to one of the boundary vertices. Call this path the `future' of $v$ (see \Cref{fig:CLEB}). We say $u $ and $v$ are in the same component of the arborescence if their futures `merge' (i.e. the futures end at a common boundary vertex). It is easy to see that this defines an equivalence relation, and hence the terminology `component' is justified.
\begin{lemma}\label{lem: arborescence_in_every_step}
    $S_j \cap \vec E_j$ is an arborescence for every $j \ge 0$.
\end{lemma}
\begin{proof}
The proof is straightforward using induction.
    Either a cycle is not created in which case we are joining two components of an arborescence. Otherwise a cycle is contracted and we are still left with an arborescence.
\end{proof}
 \paragraph{Uncontraction phase.} Since we started with a fully oriented graph, once the algorithm stops, say at step $\vartheta$, the arborescence  $\vec T_{\vartheta}:=S_\vartheta \cap \vec E_\vartheta$ contains an outgoing edge of each vertex in $V_{\vartheta}$ except $\partial$. Hence, $\vec T_{\vartheta}$ is a spanning arborescence of $(G_\vartheta, \partial)$.  During the algorithm, we also obtain an ordered sequence of cycles $ C_1, \ldots, C_{k}$ which are contracted  (step (d)) at time steps $1 \le i_1 < i_2  < \cdots < i_k \le \vartheta$, i.e., $\Psic((G_{i_{\ell}-1}, \partial), C_{\ell} )= (G_{i_\ell}, \partial)$. We iteratively uncontract the cycles in reverse order to obtain the spanning arborescences $\vec T_\ell$ of $(G_{i_\ell}, \partial)$ for $1 \le \ell \le  k$. More precisely, we set  $\vec T_{k} = \vec T_{\vartheta}$ and $$ \vec T_{\ell -1} =  \Psiuc ((G_{i_{\ell}-1}, \partial), (G_{i_\ell}, \partial), C_{\ell}, \vec T_{\ell})$$ for all $1 \le \ell \le k$. 
  
Finally, we output $\vec T^* = \vec T_0$. The next theorem shows that $T^*$ is the MSA. We reproduce the proof sketch here for the sake of completeness as this is quite short.
\begin{thm}[\cite{tarjan,RS24}]\label{thm:CLEB}
    The output $\vec T^*$ of the sequential CLEB algorithm described above produces a minimum spanning arborescence of $(G, \partial)$ for weights $w$.
\end{thm}
\begin{proof}[Proof sketch]
    We prove by backward induction on the cycles uncontracted and prove that in every step $i$, the arborescence produced by uncontraction is the MSA for the weights $w_i$.
    
    Observe that subtracting the same weight from all the outgoing edges of a vertex keeps the MSA unchanged (it is a gauge transform). The final arborescence produced at time $\vartheta$ must be the MSA of $G_\vartheta$ as its weight is 0 and the weights are non-negative in every step by design.

    Suppose we uncontract a cycle $C$ which was contracted in step $i-1$. Now take the $w_{i-1}$-weight MSA of $G_{i-1}$ and call it $T_{i-1}^*$. Our goal is to prove $T_{i-1}^*=T_{i-1}$. Observe  that in $G_{i-1}$, there is at least one vertex $v$ in $C$ with an oriented path in $T^*_{i-1}$ starting from $v$, ending at $\partial$ and containing no other vertex from $C$. Take the outgoing edge from one such $v$ in $T^*_{i-1}$, keep it and remove every other outgoing edge from $v_C$ which came from $T^*_{i-1}$. Also remove all edges in $C$ which belong to $T_{i-1}^*$. The remaining collection of oriented edges in $T_{i-1}^*$ is a spanning arborescence of $(G_i, \partial)$, call it $T'$.
Let $w_i(t)$ denote the $w_i$-weight of $t$. We have the following string of inequalities
    \begin{equation*}
        w_{i-1}(T_i) = w_{i-1}(T_{i-1}) \ge w_{i-1}(T^*_{i-1}) \ge w_{i-1}(T') \ge w_{i-1}(T_i).
    \end{equation*}
The first equality follows since the $w_{i-1}$-weight of all the oriented edges in $C$ is 0, and they are the only edges in $T_{i-1}\setminus T_i$. The first inequality is justified because $T_{i-1}^*$ is the MSA, the second inequality follows because $T'$ is obtained from $T_{i-1}^*$ by removing some edges, and the final inequality follows because gauge transform does not change the MSA. Therefore all the inequalities are actually equalities and $T_{i-1} = T_{i-1}^*$ because the MSA is unique for generic weights.
\end{proof}

\begin{remark}\label{rmk:edge_ambiguity}
Observe that since the vertex set can potentially change in each step of the CLEB process, the head or tail of an oriented edge $\vec e$ in $\vec E_j$ can have ambiguous meaning if $\vec e \in \vec E_i$ for some $i <j$ as well. Therefore, whenever we speak of the head or tail of an oriented edge, we need to specify the underlying graph. We define the head (resp. tail) of $\vec e$ in $G_i$ to be the head (resp. tail) of $\vec e$ in $V_i$ when we think of $\vec e \in \vec E_i$.  
\end{remark}
\begin{remark}
Let $Z_1 \subset V_i$ and $Z_2 \subset V_j$.
Let $\vec e$ be the edge exposed by the CLEB process in step $k$.
    We say the CLEB process in step $k\ge \max\{i,j\}$ \textbf{hits} $Z_2$ from $Z_1$ if the tail of $\vec e$ in $V_i$ is in $Z_1$ and the head of $\vec e$ in $V_j$ is in $Z_2$. 
\end{remark}
\subsection{The CLEB walk algorithm}\label{sec:CLEB_walk_new}
We recall that in the sequential CLEB algorithm,  the choice of the vertex in step (b) can depend on the weights, and may involve additional randomness as well. This property will be exploited heavily. For future reference, we now define a useful choice of ordering for the sequential CLEB which we call the \textbf{CLEB walk algorithm} (for the experts, it is  reminiscent of the celebrated  Wilson's algorithm.) This idea was first proposed by Gabow, Galil, Spencer and Tarjan \cite{gabow86}.\ As Wilson's algorithm performs successive loop erased random walks, \textbf{CLEB walk algorithm}  is executed by performing successive \textbf{CLEB walks}, and instead of erasing loops, the CLEB walk contracts a loop.

Suppose we have a finite oriented multi-graph $(G, \partial)$ with boundary $\partial$ and with generic weights. Suppose we are in an intermediate step of the CLEB walk algorithm, that is, we have the collection $(H_i)_{0 \le i \le j-1}$. At this point, we can decide to pick a vertex $x \neq \partial $ in $V_{j-1}$ which has no outgoing exposed edge at step $j-1$. The \textbf{CLEB walk started at  $x$} is generated by running the sequential CLEB algorithm on $(G_{j-1}, \partial)$ with weights $(w_{j-1,\vec e})_{\vec e \in \vec E_{j-1}}$  for a particular choice of vertices until a certain stopping time $\tau$ which we describe below.

For time step $j$, we choose the vertex $v$ (in step (a)) to be $x$. Then inductively in step $k \ge j$, we choose the vertex as follows: 
\begin{itemize}
\item [a.] If a cycle $C$ is contracted in $G_{k-1}$ to produce $G_k$, then choose the vertex $v_C$ which is the new vertex in $G_k$ obtained by contracting the vertices of $C$. Expose the minimum weight outgoing edge from $v_C$, and if a cycle is created by this exposed edge, contract it. 
\item [b.] Otherwise, if the head of the exposed edge in step $k$ is $\partial$, declare $\tau =k$ and stop the walk. 
\item [c.] Otherwise, if the head of the exposed edge in step $k$ is a vertex which already has an exposed outgoing edge in $S_{k-1}$, stop the walk as well. We emphasize here that we employ this step only if step a. is not applicable.
\item [d.] Otherwise, choose the head of the exposed edge in step $k$. Expose the minimum weight outgoing edge from this vertex, and contract a cycle if created.
\end{itemize}
For any $t <\tau-j+1$, a CLEB walk started at $x$ and run for $t$ steps is exactly the above algorithm run up to step $k=j-1+t<\tau$.

The CLEB walk described above is a slight generalization of  the one considered in \cite{RS24}. There, the only case considered was the CLEB walk with initial condition $H_0 $ (i.e., nothing is exposed). For our case, we need to consider a situation where we have a collection of exposed edges which is not necessarily connected to the boundary. Nevertheless, it turns out that item c. never arises in the special case of $(H_i)_{0 \le i \le j-1}$ that we need to consider. We still keep item c. for the sake of keeping the process well defined.

\subsection{Summary of notations}
\begin{itemize}
   \item $S_j:=$ the set of exposed edges up to step $j$ in sequential CLEB algorithm.
\item $W_{j,\vec e}:=$ the weight of an edge $\vec e$ after step $j$ in the sequential CLEB algorithm.
\item $\cO(v) := \{\vec e:\vec e_- = v\}$= the set of outgoing edges from $v$.
\item $\cO(v, t):=$ the set of outgoing edges from $v$ in $\vec E_t$. 
\item $\cO(v, t, S_t^c):=$ the set of oriented edges in $\cO(v,t) $ which is not in $S_t$ (the set of unexposed outgoing edges from $v$ in $\vec E_t$).
\item $\vec e_{\min}(v,t):=$ the oriented edge with minimum weight among those in $\cO(v, t, S_t^c)$. 
\item  $\vec e_{\min}(v)$ is the oriented edge with minimum $W$-weight among $\cO(v)$. That is, $$\vec e_{\min}(v) = \{\vec e \in \cO(v): W(\vec e) < W(\vec f) \text{ for all }\vec f \neq \vec e, \vec f \in \cO(v)\}.$$
We call $\vec e_{\min}(v)$ the \textbf{minimum weight outgoing edge from $v$.}
\item We say the CLEB process \textbf{hits} $S_2$ from $S_1$ if the tail of the exposed edge $\vec e$ is in $S_1$ and the head of $\vec e$ is in $S_2$. 
\item $\cN_j$:= roughly, it denotes the number of vertices exposed by the CLEB process up to step $j$. More formally, it is the head of the edges in $S_j$ union the first vertex selected in the process (the tail of $S_1$). 
\item $N_j :=|\cN_j|$. 
\item  $\tilde \cN_j:=$ the set of heads of  $ S_j \cap \vec E_j$, that is, $\tilde N_j$ is the set of heads of the exposed edges in the contracted graph after step $j$.
\item We say the CLEB process in step $k\ge \max\{i,j\}$ \textbf{hits} $Z_2$ from $Z_1$ if the tail of $\vec e$ in $V_i$ is in $Z_1$ and the head of $\vec e$ in $V_j$ is in $Z_2$.
\item $\cont(v):=$ the set of vertices contracted to create $v$.
\item $\cont(A)=\cup_{v\in A}\cont(v)$.
\item $\Tip_i$: $\cont(X_i)$ or the set of vertices contracted to create the tip of the path $P_i $ in the local CLEB process.
\item $\tip_j$:= $|\Tip_j|$.
\item $\tau_1,\tau_2,\dots$ are times when `past' is hit, i.e., times in $\hitpast$. 
\item $\alpha_1,\alpha_2,\dots$ are times when ${\sf f}$ is hit, i.e., times in $\hittree$.
\item $\cB_k:=$ the set of blue vertices up to step $k$.
\item $B_k:= |\cB_k|$.
\item $S_{\le k}:=(S_i)_{0\le i\le k}$.
\item $\mathfrak B:=$ bad event in stage 2.
\item $\mathfrak D:=$ bad event in stage 3.
\item $A_n=\log\log n$ and $\eps_n=A_n^{-1}$.
\item We use the notation $a_n = O(b_n)$ to mean that there exists a constant $C>0 $ such that  $a_n/b_n \le C$ for all $n \ge 1$. We also use the notation $a_n/b_n = o(1)$ to mean that $a_n/b_n \to 0$ as $n \to \infty$.
\end{itemize}

\section{Analysis of the CLEB process}\label{sec:proof_main}
From here on, we work with the fully oriented complete graph $\vec K_n = ({\sf V}_n, \vec {\sf E}_n)$. Sometimes we drop the suffix and call them ${\sf V},\vec {\sf E}$ for notational clarity. Throughout, fix a vertex labeled $1 (\neq \partial)\in {\sf V}$. Recall that we assign i.i.d. Exponential(1) weight $W(\vec e)$ to each oriented edge $\vec e\in\vec {\sf E}$. 
We employ the sequential CLEB algorithm, and 
the key is to find the correct sequence of vertices to expose to aid in the proof of \Cref{thm:main}. We perform the contraction phase of the algorithm in three stages. We refer to \Cref{sec:outline} for motivations behind the following choice of running the sequential CLEB algorithm. Let's call this special choice of sequential CLEB \textbf{local CLEB process}, which we describe below.

Let $\cN_j$ be the set of vertices exposed by the local CLEB process up to step $j$ (assuming $\cN_0 = \{1\}$). More formally, it is the head of the edges in $S_j$ union the first vertex selected in the process (the tail of $S_1$). We define $N_j=|\cN_j|$. If $x\in \R$, we define $\cN_x = \cN_{\lfloor x \rfloor}$ and $N_x = |\cN_x|$.

\begin{description}
    \item[Stage 1: Expose the incoming arborescence.]
Fix $R \ge 1$.
We run the CLEB walk started from $1$ for $R+1$ steps. If at $t$-th step $(1<t \le R+1)$ the exposed edge hits $\cN_{t-1}$ (that is, the exposed edge hits a vertex already exposed in the past), we \textbf{exit} the process.

Suppose we have not exited the process. Then $\cN_{R+1}$ is a simple oriented path of length $R+1$. Now we want to expose the `minimum weight incoming tree' into the vertices of $\cN_R$ which we describe now (we choose $\cN_R$ and not $\cN_{R+1}$ for a small technical reason which will be clear later).
Informally, for every vertex $v$ in $\cN_R$, we expose all the minimum weight outgoing edges whose head is $v$. Then for every $\vec e$ added in the previous step, we expose all the minimum weight outgoing edges whose head is $\vec e_-$. We iterate this process until we can (see \Cref{badevents}).

More formally, define ${\sf F}^{(0)}_{n}=\mathcal \cN_{R}$ and for some $i \ge 0$, if ${\sf F}^{(i)}_{n} \neq \emptyset$, then
$$
{\sf F}^{(i+1)}_{n}=\{v\in {\sf V}_{n}\setminus\displaystyle\bigcup_{0\leq j\leq i} {\sf F}^{(j)}_{n} : (\vec{e}_{\min}(v))_{+}\in {\sf F}^{(i)}_{n}\}.
$$
If ${\sf F}^{(i)}_{n}=\emptyset$, define ${\sf F}^{(j)}_{n}=\emptyset$, for all $j>i$. Let 
$${\sf F}_{n}=\displaystyle\bigcup_{i\geq 0} {{\sf F}}^{(i)}_{n}.$$ 
Now the description of the choices of vertices in the sequential CLEB algorithm for this stage is straightforward:
choose the vertices of ${\sf F}_n \setminus {\sf F}_n^{(0)}$ in any order. Observe that by choice, no cycle is created for any such step. This completes stage 1 of the contraction phase. Let $\tauendone$ be the time when stage 1 is complete.
Define $\Upsilon_n = S_{\tauendone}$ (the set of exposed edges at the end of stage 1).
We observe that $\Upsilon_n$ has a very simple description. Namely, $\Upsilon_n$
is $\cN_{R+1}$ along with a union of $(R+1)$-disjoint arborescences, each with its boundary being a single vertex in $\cN_R$. Name the arborescence with boundary $v \in \cN_R$, $\Upsilon_{n,v}$. Call $\Upsilon_n$ the \textbf{incoming arborescence}.

\begin{figure}[h]
    \centering
    \includegraphics[width=0.7\linewidth]{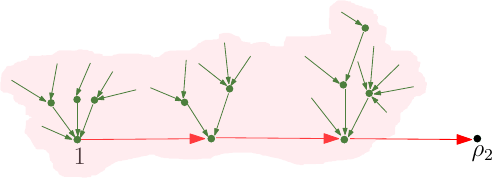}
    \caption{An illustration of stage 1 of the local CLEB process. Suppose $R=2$ and  $\cN_3$ is given by the red oriented edges. Then incoming arborescence $\Upsilon_n$ is given by the green and red oriented edges. The vertices in the pink region are in ${\sf F}_n$.}
    \label{badevents}
\end{figure}

We move to stage 2 only if we have not exited the process. 
    
    \item [Stage 2: CLEB walk from tip of $\cN_R$.] 
Let $\rho_R$  be the  head of the last edge exposed in $\cN_{R+1}$. Perform CLEB walk as described in \Cref{sec:CLEB_walk_new} started at $\rho_R$. Let $\tauend$ be the stopping time when stage $2$ is complete (i.e. the CLEB walk started from $\rho_R$ hits $\partial$).
    \item [Stage 3: CLEB walk from rest.]
    
    Let $\tilde \cN_j$ be the union of the set of heads and tails of  $ S_j \cap \vec E_j$. That is, $\tilde \cN_j$ is the set of heads and tails of the exposed edges of the contracted graph at step $j$.
    Order the vertices in ${\sf V}_n \setminus \{\partial\}$ that are not exposed by the end of  stage 2 in any arbitrary order. Perform CLEB walk as described in \Cref{sec:CLEB_walk_new} from the first vertex until $\tilde {\cN}_{{\tauend}}$ is hit.
    Suppose this happens at time $\zeta_1$. Then pick the next vertex in the ordering which is not exposed by time $\zeta_1$ and is not $\partial$ and perform the CLEB walk algorithm from that vertex until $\tilde \cN_{\zeta_1}$ is hit. Continue until there is no vertex left to be exposed except $\partial$.
\end{description}
The uncontraction process is now performed as usual to output the MSA.

\bigskip
We point out that stages 2 and 3 are reminiscent of the celebrated Wilson's algorithm \cite{wilson_algo} which is used to sample the uniform spanning tree, except instead of erasing loops we contract it.
Notice that stage 1 of the above algorithm comes with a choice of $R$ (which we will choose later to be large but fixed).
For future use, define $W_j$ to be the collection of weights $(W_{j,\vec e})_{\vec e \in \vec E}$ after step $j$ of the local CLEB algorithm. 

\subsection{Analysis of Stage 1}
We start with a simple lemma.
\begin{lemma}\label{lem:exit}
    Probability that the local CLEB process exits in stage 1 tends to 0 as $n \to \infty$.
\end{lemma}
\begin{proof}
Note that,
\begin{equation*}
\{\text{the process exits in stage $1$}\} \subseteq \cup_{i=1}^{R+1}\{\text{a cycle is created in the $i$-th step}\}.
\end{equation*}
In every step, the minimum weight outgoing edge is uniformly distributed among all the outgoing edges (since the weights are i.i.d.). Also the number of vertices exposed in $i$ steps is at most $i$. Thus a cycle is created in the $i$th step with probability at most $i/(n-1)$.
Therefore by union bound,
\begin{align*}
\P(\text{the process exits in stage $1$})&\leq\P(\cup_{i=1}^{R+1}\{\text{a cycle is created in the $i$-th step}\})&\\
&\leq\sum_{i=1}^{R+1}\frac{i}{n-1}\to 0
\end{align*}
as $n\to\infty$ and we are done.
\end{proof}
In the proof of \Cref{lem:exit}, we simply use the fact that the minimum weight outgoing edge is uniformly distributed until a cycle is created. This is true for any continuous distribution. After a cycle is created, the contracted vertex has a potentially complicated distribution of the weights of the outgoing edges because of the weight subtraction step in the CLEB algorithm (item (b) in the contraction phase). However, if the weights are i.i.d.\ Exponential$(1)$ then, conditioned on the exposed edges and their weights, the uniformity of the minimum weight outgoing edge from every vertex is still preserved (for the subtracted weights). This is a simple consequence of the memoryless property of Exponential random variables.

To analyze stage 2, we need to exploit the memoryless property even more. Fix a set of vertices $S$ and suppose $U \subseteq {\sf V} \setminus S$ be the set of vertices such that the minimum weight outgoing edge from each vertex $u\in U$ hits $S$. Then conditioned on $U$ and the weights of $(\vec e_{\min} (u))_{u \in U}$, the minimum weight outgoing edge from every vertex $z \in {\sf V} \setminus (S \cup U)$ is uniform among all the edges in $\cO(z)$ which do not hit $S$. Let $\pi_z $ be the weight of the minimum weight outgoing edge from $z$. For each $z \in {\sf V} \setminus (S\cup U)$, after subtracting $\pi_z$ from the weight of every oriented edge in $\cO(z)$, the law of the weights of the edges in $\cO(z)$ except $\vec e_{\min }(z)$ remain i.i.d.\ Exponential$(1)$. We now state and prove an elementary lemma about exponential random variables which summarizes this property.

\begin{lemma}\label{lem:exp}
    Suppose we have $n$ i.i.d.\ Exponential(1) random variables $X_1, X_2,...., X_n$. Let $S\subsetneq\{1,2,...,n\}$ and $X_J=\min\{X_1,X_2,...,X_n\}$. Then
    \begin{enumerate}[a.]
    \item The distribution of $J$ conditioned on $J \not \in  S$ is Uniform$(\{1,2,...,n\}\setminus S)$.
        \item Conditioned on $J, X_J$, the conditional joint distribution of $(X_i-X_J)_{i\neq J}$ is i.i.d. Exponential$(1)$.
    \end{enumerate}
\end{lemma}
\begin{proof}
Since $S$ is a fixed subset and $X_1,X_2,...,X_n$ are i.i.d\ continuously distributed random variables, $J$ is uniform over $\{1,2,...,n\}\setminus S$. This proves part a.
Part b.\ is a simple consequence of the memoryless property of Exponential random variables.
\end{proof}

Define a critical Galton--Watson arborescence by taking a critical Galton--Watson tree and then orienting every edge towards the parent. Clearly, the common ancestor of every vertex is the boundary of this arborescence. We now prove that the arborescences exposed in stage $1$ converge to i.i.d. Poisson$(1)$ Galton--Watson arborescences.

\begin{prop}\label{prop:incoming_tree}
Condition on the event that we do not exit in stage $1$.
Then $(\Upsilon_{n,v})_{v \in \cN_R}$ converges to $(R+1)$ many i.i.d.\ Poisson(1) Galton--Watson arborescences  as $n \to \infty$. For each $v \in \cN_R$, the limit of $\Upsilon_{n,v}$ is an arborescence with boundary $v$. Furthermore, $\rho_R \not \in {\sf F}_n$ with probability tending to $1$ as $n \to \infty$.
\end{prop}
\begin{proof}
Order the vertices of $\cN_R$ in some order. Then we do breadth first search exploration of $\Upsilon_{n,v}$  for  $v\in \cN_R$ according to this order. In a generic step, when we explore a vertex $u$, we expose all the vertices each of whose minimum weight outgoing edge hits $u$ (which includes $\rho_R$ if it is not part of the exposed vertices). Suppose we have done $k$ steps of the exploration and we have shown that the vertices exposed in these steps converge to i.i.d.\ Poisson$(1)$ and $\rho_R$ is not in the set of exposed vertices with probability tending to $1$ as $n \to \infty$. Let $X_{n,k}$ be the number of vertices exposed up to $k$ steps (which converges by our hypothesis to a sum of $k$ i.i.d.\ Poisson$(1)$). Now let us  analyze step $k+1$ and suppose we are exploring vertex $v$. Pick any unexplored vertex $u$. We know that the minimum weight outgoing edge from $u$ cannot hit the $k$ vertices which are already explored. Therefore by \Cref{lem:exp} part a., the probability that the minimum weight outgoing edge from this vertex hits $v$ is $(n-1-k)^{-1}$.  The number of unexposed vertices is $n-X_{n,k}$. Thus the number of incoming edges to $v$ converge to Poisson$(1)$ provided $X_{n,k}$ is tight in $n$. The tightness of $X_{n,k}$ follows from our assumption that they converge to i.i.d. Poisson$(1)$. The tightness also shows that the probability of $\rho_R$ hitting $u$ tends to 0 as $n \to \infty$.  This concludes the proof. 
\end{proof}

In order to analyze the CLEB process after stage 1, we need to understand the conditional distribution of the weights at the end of stage 1.
More precisely for a fixed and appropriate ${\sf t} , {\sf f}$, a collection of positive numbers $a_{\vec e}$ for all $\vec e \in {\sf t}$, let $$\mathcal{IT}_{n} = \mathcal{IT}_n ({\sf t};a_{\vec e} ,\vec e \in {\sf t} ):=\{\Upsilon_n = {\sf t}, {\sf F}_n = {\sf f}, W(\vec e) = a_{\vec e} \text{ for all }\vec e \in {\sf t}, \rho_R \not \in {\sf f}\}.$$
It is implicit in the definition of $\mathcal{IT}_n$ that the process did not exit.
Condition on $\mathcal{IT}_{n}$. Let us analyze the distribution of the weights of the outgoing edges from  ${\cX}:={\sf V} \setminus \{\cN_{\tauendone} \setminus \{\rho_R\}\}$ (the set of vertices which is either $\rho_R$ or is unexposed by time $\tauendone$). For every vertex in $\cX$, the minimum weight outgoing edge from this vertex cannot have its head in ${\sf f}$. Also, obviously, no weight is subtracted from any unexposed vertex up to time $\tauendone$. Therefore, the joint law of the weights of edges in $\cO(v)$ for any unexposed vertex $v$ is i.i.d.\ Exponential$(1)$ conditioned on the event that the minimum weight outgoing edge from $v$ does not have its head in ${\sf f}$. 

To analyze stage 2, we need to know two things. If the CLEB walk hits an unexposed vertex, we need to know the conditional distribution of the minimum weight outgoing edge from this vertex. This is uniform as explained just above. However, to analyze the process further down the road, we also need to know the joint distribution of the weights when we subtract this minimum weight. 
For exponential weights this turns out to be quite nice as posited by the following corollary, which is a consequence of \Cref{lem:exp}.

Since the graphs and the weights are changing in every step, we first need to get through some notational barriers. Recall $\tilde \cN_t$ is the set of exposed vertices in the contracted graph after $t$ steps of stage 2, i.e., $\tilde \cN_t$ is the union of heads and tails of $S_t \cap \vec E_t$. Let us introduce some related notations. For a vertex $v \in V_t$, let $\cO(v, t)$ be the set of outgoing edges from $v$ in $\vec E_t$.  Let $\cO(v, t, S_t^c)$ be the set of oriented edges in $\cO(v,t) $ that is not in $S_t$ (there can be at most one outgoing edge in $\vec E_t$ for any $v \in V_t$). For $v \in \tilde \cN_t \setminus {\sf f}$, let $\cO(v,t,{\sf f})$ be the set of outgoing edges from $v$ in $\vec E_t$ whose head in the original vertex set ${\sf V}$ is in ${\sf f}$ (recall \Cref{rmk:edge_ambiguity}). Suppose for every $v \in V_t$, $\vec e_{\min}(v,t) $ is the minimum $W_t$-weight unexposed outgoing edge from $v$, i.e., $\vec e_{\min}(v,t)$ is the oriented edge with minimum $W_t$-weight among those in $\cO(v, t, S_t^c)$ (we add the notation $t$ to emphasize the dependence on the weights $W_t$).
Also recall that in every step of the CLEB walk, in a generic step $t+1$, we always choose the `tip' of the exposed path up to step $t$. For any $t>\tauendone$, let $\cE_t$ denote $S_t $ and the weights of the oriented edges in $S_t$. Also, the tip chosen in step $t+1$ is in ${\sf V}$ means that in step $t$ an edge is chosen whose head is unexplored. If the tip is not in ${\sf V}$, then a cycle is created in step $t$.

\begin{corollary}\label{cor:cond_law}
   Condition on $\mathcal{IT}_n ({\sf t};a_{\vec e} ,\vec e \in {\sf t} )$ and $\cE_t$ and assume  $ t \ge \tauendone$.  Let $z$ be the vertex which is to be chosen to explore in step $t+1$, and we assume such an unexplored vertex exists. The following is true almost surely on the conditioning.
   \begin{enumerate}[a.]
   \item For every $v$ in $V_t \setminus \tilde \cN_t $, $\vec e_{\min}(v,t) $ is uniform in $\cO(v,t) \setminus \cO(v,t,{\sf f})$
       \item If $z \in {\sf V}$, then $\vec e_{\min}(z,t) $ is uniform in $\cO(z,t) \setminus \cO(z,t,{\sf f})$. Furthermore, the joint distribution of the $W_t$ weights of $\{\cO(v, t, S_t^c): v \in \tilde \cN_t \setminus \{z\}\} $ is i.i.d.\ Exponential$(1)$
       
      \item If $z \not \in {\sf V}$, then $\vec e_{\min}(z,t) $ is uniform in $\cO(z,t)$. Furthermore, the joint distribution of the $W_t$ weights of $\{\cO(v, t, S_t^c): v \in \tilde \cN_t\} $ is i.i.d.\ Exponential$(1)$.    
   \end{enumerate}
\end{corollary}
\begin{proof}
    We proceed by induction on $t$. For $t = \tauendone$, $z = \rho_R$ which is in ${\sf V} $ since we have not exited the process. In this case parts a. and b. are simple consequences of \Cref{lem:exp}. Indeed,  if $v \in (V_t \setminus \tilde \cN_t) \cup \{z\}$, then we apply \Cref{lem:exp} part a. (with $S = \cO(v,t,{\sf f})$) Furthermore, for the joint distribution of the $W_t$ weights of $\{\cO(v, t, S_t^c): v \in \tilde \cN_t \setminus \{z\}\} $, we apply \Cref{lem:exp}, part b.

    Now suppose the lemma is true at time $t>\tauendone$. If there is no vertex left to explore, we are done, so suppose such a vertex $z$ exists.  In step $t+1$, we explore $z$ and suppose we expose $\vec e$. Part a. follows immediately from induction since $\tilde \cN_{t+1} \supseteq \tilde \cN_t$ and since exploring $z$ do not affect the weights of any oriented edge whose tail is not in $z$.

   For the remaining parts, observe that $z$ may or may not be in ${\sf V}$, but irrespective of that, \Cref{lem:exp}, part b. tells us that after exposing $\vec e$, the joint law of $$\{W_t(\vec f)  - W_t(\vec e): \vec f \in \cO(z,t,S_t^c) , \vec f \neq \vec e\}$$
    is i.i.d.\ Exponential$(1)$. Therefore, if $\vec e$ creates a cycle, then we are in the situation of part c.  and hence the lemma follows from induction and the above observation.

   Now  suppose that $\vec e_+$ is a previously unexplored vertex, which necessarily is in ${\sf V}$ and hence we only need to prove part b. The first part of part b. follows from part a. of induction.
   The second part of part b. follows from \Cref{lem:exp}, part b. and the same observation as in the previous case.
\end{proof}

\subsection{Analysis of Stage 2 and 3.}\label{sec:2 and 3}

Define 
\begin{equation}
  \P^{\sf t}(\cdot ) := \P(\cdot | \mathcal{IT}_n ({\sf t};a_{\vec e} ,\vec e \in {\sf t} ))\label{eq:Pt}  
\end{equation}
to be the conditional law after stage 1 given that $\sf t$ and the weights on the edges of ${\sf t}$ is exposed and $\rho_R\notin{\sf f}$.

First we need a lemma which says that we can immediately uncontract after stage 2 and we don't need to wait until stage 3 to obtain the MSA.

\begin{lemma}\label{lem:stage2_and_3flip}
    Let $C:=(C_1,\ldots, C_m)$ be the sequence of cycles contracted in stage $2$. Suppose $D:=(D_1,\ldots, D_p) $ is the sequence of cycles contracted in stage 3. Then uncontracting the cycles in $C$ in reverse order first and then $D$ in reverse order also produces the MSA. 

    Uncontracting the cycles in $C$ in reverse order produces a spanning arborescence of the graph spanned by $\cN_{\tauend}$ with boundary $\partial$.
\end{lemma}
\begin{proof}
    In stage $3$ no cycle is contracted which contain an edge in $S_{\tauend}$. Thus the uncontraction procedure in the two stages can be done `independent' of each other.  A more general version of this lemma is proved in \cite[Lemma 4.2]{RS24}
\end{proof}

We need some definitions to set up the analyses of stages 2 and 3.
For every vertex $v$ in $G_j$, we define $\cont(v)$ to be the set of vertices which were contracted to create $v$. More formally, for every $v \in V_0 ={\sf V}$, we define $\cont(v)=v$. Having defined $\cont(v)$ for all $v \in V_j$, we inductively define $\cont(v), v\in V_{j+1}$ as follows. If no cycle is contracted in step $j+1$, $V_{j+1}=V_j$ and $\cont(v)$ is unchanged for every $v \in V_{j+1}$. If a cycle $C$ is contracted in step $j+1$, we keep $\cont(v)$ unchanged for every $v \in V_{j+1}\setminus\{v_C\}$. Let $U$ be the set of vertices contracted in step $j+1$ to create $C$ (i.e. $U$ is the heads of the oriented edges in $C$). Now define $\cont(v_C)= \cup_{u \in U}\cont(u)$. Observe that if a vertex $v$ is in $V_i$ as well as $V_j$ then $\cont(v)$ is the same for both, and hence there is no ambiguity in the  definition of $\cont$ (in the sense that we do not need to specify which graph $v$ belongs to). To keep track of the `tip' after $i$-th step, let's define the following sequence $(X_i)_{i\geq\tauendone}$. $X_{\tauendone}:=\rho_R$. If in $i$-th step, CLEB walk hits a previously unexposed vertex $v\in{\sf V}$, define $X_i=v$. If in $i$-th step, CLEB walk contracts a cycle $C$, then define $X_i=v_C$. Note that $X_i$ is always a vertex in $G_i$. We will call $X_i$ the \textbf{tip} of the CLEB walk after step $i$. If $A\subset V_j$, for some $j$, $\cont(A):=\cup_{u\in A}\cont(u)$. Define $\Tip_j:= \cont(X_j)$ and $\tip_j:=|\Tip_j|$ for all $j\geq\tauendone$. In other words, $\Tip_j$ is the set of vertices in ${\sf V}$ which were contracted (perhaps multiple times) to create the `tip' after $j$-th step.

Recall that $\tauend$ is the stopping time when stage 2 of the local CLEB process is completed, i.e., $\tauend$ is the first time the CLEB walk from $\rho_R$ hits $\partial.$ 
Our plan is to define a bad event and show that if the bad event does not happen, then after the uncontraction process, the arborescence ${\sf t}$ is `undisturbed' in the sense that no edge is added whose head is in ${\sf f}$ and no edge of ${\sf t}$ is erased.
 First of all, if the CLEB walk does not hit ${\sf f}$ then ${\sf t}$ is not part of any cycle and is unchanged. However recall from the heuristics of the proof outline in \Cref{sec:outline}  that, in fact, ${\sf f}$ will be hit with positive probability.

\begin{figure}[t]
    \centering
    \includegraphics[width=0.6\linewidth]{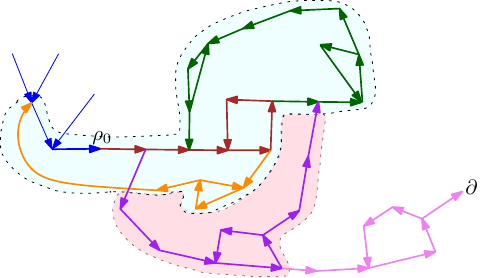}
    \caption{An illustration of $\hitpast$ and $\hittree$ in stage 2 with $R=0$. The incoming tree ${\sf t}$ is colored blue. The edges exposed in steps $(\tauendone,\tau_1]$ are colored brown, those in $(\tau_1,\tau_2]$ are colored green, those in $(\tau_2,\tau_3]$ are colored orange, those in $(\tau_3,\tau_4]$ are colored purple, and the final excursion $(\tau_4,\tauend]$ is colored violet ($J_{\max}=4$). Here $\alpha_1=\tau_3$. At time $\alpha_1=\tau_3$, all the vertices inside the light cyan area are contracted into a single vertex $X_{\tau_3}$, and at time $\tau_4$, the vertices in the pink and the light cyan area are contracted into the vertex $X_{\tau_4}$. 
    The blue edges outside the light cyan  form ${\sf t}_{\tau_3} = {\sf t}_{\tau_4}$. }
    \label{fig:excursion}.
\end{figure}

 But ${\sf t}$ can remain unchanged after uncontraction even if ${\sf f}$ is hit and to describe that we need a finer analysis of the uncontraction process. We now define the successive times when the local CLEB process hits its `past'.
 More formally, let $$\tau_1 = \min\{\tauendone\le t<\tauend : \text{local CLEB process hits $\cN_{t-1}$ in step } t\}.$$ Inductively, having defined $\tau_j$ for $j \ge 1$, define $$\tau_{j+1} = \min \{\tau_j< t< \tauend: \text{local CLEB process hits $\cN_{\tau_j}$ in step $t$}\}.$$
Set the convention $\inf \emptyset = \infty$, which implies that if $\tau_j=\infty$ then $\tau_{m}=\infty$ for all $m>j$. Let $$\hitpast := \{\tau_j: \tau_j <\infty\}$$ be the set of times when the `past' is hit. Observe that the vertices in $\cont(X_{\tau_j})$ is `eaten' by the vertices in $\cont(X_{\tau_{j+1}})$ (i.e. $\cont(X_{\tau_j}) \subset \cont(X_{\tau_{j+1}})$). This is related to the `snowball' effect as described in the proof outline (\Cref{sec:outline}).

Now we define the sequence of stopping times when the local CLEB process hits ${\sf f}$ in stage $2$. Let $\alpha_0=\tauendone$ and $$
\alpha_j=\inf\{\alpha_{j-1}<t\leq\tauend:\text{the local CLEB process hits ${\sf f}$ at time $t$}\},$$
for all $j\geq 1$, where we adopt the convention that $\inf \emptyset = \infty.$ Let $\hittree = \{\alpha_j: \alpha_j <\infty\}$. Note $$\hittree\subseteq \hitpast.$$

Let us spend a little time analyzing the geometry of the CLEB walk at the times $\hitpast$ and $\hittree$. Now observe that at any time $\tau_j\in \hitpast \cap (\alpha_1,\infty)$, it must be the case that either $\cont(X_{\tau_{j-1}})$ or ${\sf f}$ is hit. Therefore, at time $\tau_j$, all the vertices exposed between $\tau_{j-1}$ and $\tau_{j}$ become part of $\cont(X_{\tau_{j}})$. Also, all the vertices ever exposed up till time $\tau_j$ are all in $X_{\tau_j}$ (see \Cref{fig:excursion}). Consequently, the exposed edges in the contracted graph $G_{\tau_j}$ (i.e. $S_{\tau_j} \cap \vec E_{\tau_j}$) is an arborescence ${\sf t}_{\tau_j}$ with a single boundary vertex in $V_{\tau_j}$ which is $X_{\tau_j}$. 
If $k=\tauendone$, then $X_k = \rho_R$.
Let $J_{\max}$ be the largest $j$ so that $\tau_j<\infty$. We call the time intervals $(\tauendone,\tau_1],(\tau_1,\tau_2],\cdots,(\tau_{J_{\max}},\tauend]$ \textbf{excursions}. We call the interval $(\tau_{J_{\max}}, \tauend]$ the \textbf{final excursion}. 
\begin{figure}[h]
    \centering    \includegraphics[scale = 1.0]{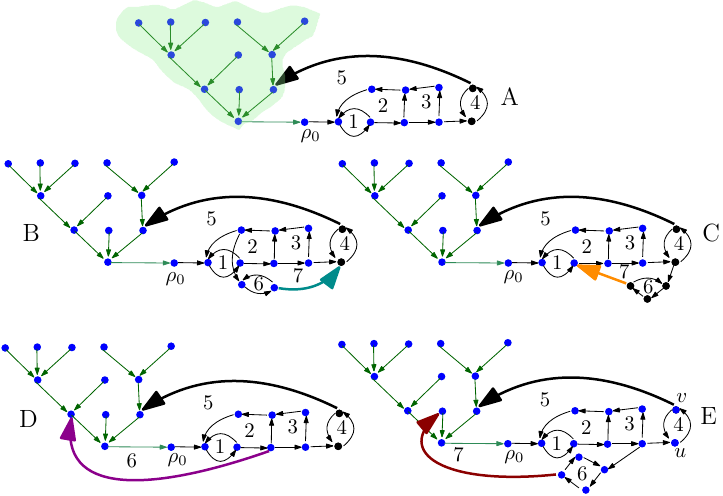}
    \caption{We illustrate the growth of blue vertices in the local CLEB process here. In A, the incoming arborescence ${\sf t}$ (colored green) is attached to $\rho_0$. Local CLEB process starts from $\rho_0$ and creates some cycles numbered chronologically as $1,2,3$ at times $\tau_1,\tau_2,\tau_3$. Then it creates $4$-th cycle and hits back the tree for the first time. So that $\alpha_1=\tau_4$ and we color $\cN_{\alpha_1}\setminus\Tip_{\alpha_1-1}$ blue. In B, at $\tau_4+1$, the process goes out from a blue vertex and at $\tau_5$ it doesn't hit ${\sf f}$. It creates cycles $6$ and $7$ in $(\tau_4,\tau_5]$. Thus the set of blue vertices grows. In C, at $\tau_4+1$, it goes from a non-blue vertex and doesn't hit ${\sf f}$ at $\tau_5$. This keeps the set of blue vertices unchanged. In D, at $\tau_4+1$, the local CLEB process goes out from a vertex and hits ${\sf f}$ immediately, which again keeps the set of blue vertices unchanged. In E, after $\tau_4$, it hits ${\sf f}$ after creating the $6$-th cycle. This is very bad event for our purposes and we make every vertex blue in this case, including $u$ and $v$.}
    \label{fig:Bad final 2}
\end{figure}

We now inductively color some of the vertices blue. The blue vertices will track a certain set of vertices which we will eventually see are `bad' in the following sense: if the final excursion does not start at blue then ${\sf f}$ is guaranteed to be `undisturbed' after the uncontraction process. Our final goal would be to prove later that the probability that the last excursion starts at a blue vertex is small, or in effect, the set of blue vertices do not grow very much with high probability.

With this in mind, 
we now define an increasing sequence of blue vertices $(\cB_k)_{\tauendone \le k \le \tau_{J_{\max}}}$ when step $k$ is completed. Define $\cB_{\tauendone} = {\sf f}$, that is, at the beginning of stage $2$, we make every vertex of ${\sf f}$ blue. Let $B_k=|\cB_k|$. 
\begin{itemize}
    \item 
If $\alpha_1=\infty$ (i.e. ${\sf f}$ is never hit), the set of blue vertices remain unchanged forever. That is, $\cB_k={\sf f}$, for all $\tauendone\leq k\leq\tau_{J_{\max}}$.
\item For $k<\alpha_1$, no vertex other than those in ${\sf f}$ is colored blue, i.e. $\cB_k ={\sf f}$.
\item If $\alpha_1<\infty$, color every vertex except the tip at step $\alpha_1$ blue. That is, color $\cN_{\alpha_1} \setminus \Tip_{\alpha_1-1}$ blue (see \Cref{fig:Bad final 2}, A).
\item Suppose we have defined the set of blue vertices $\cB_{\tau_j}$ for $\tau_j \ge \alpha_1$. Consider the next excursion $(\tau_j, \tau_{j+1}]$ and suppose it's not the final excursion. Then there could be several cases. Any such excursion `starts' and `ends' at a vertex in $\cN_{\tau_j}$. More formally, let $\vec e$ be the edge exposed in step $\tau_j+1$ and let $\vec f$ be the edge exposed in step $\tau_{j+1}$. Then define the starting point of the excursion to be the vertex $\vec e_-$ in ${\sf V}$ and the ending point of the excursion to be the vertex $\vec f_+$ in ${\sf V}$. 
\begin{itemize}
    \item If ${\sf f}$ is hit immediately, change nothing. More formally, if  $\tau_{j}+1  = \tau_{j+1} \in \hittree$ then keep the set of blue vertices unchanged (see \Cref{fig:Bad final 2}, D).
    \item Otherwise, we have several cases.
    \begin{itemize}
\item [a.] If the ending point of the excursion is in ${\sf f}$ (that is, $\tau_{j+1} \in \hittree$ and $\tau_{j+1}>\tau_j+1$) then we color everything blue for all future times up until $\tau_{J_{\max}}$, i.e. $\cB_k= \cN_k$ for all $\tau_{J_{\max}} \ge k \ge \tau_{j}+1$. (\Cref{fig:Bad final 2}, E)
    \item [b.] If the starting point of the excursion is not blue and the ending point is not in ${\sf f}$ then keep the set of blue vertices unchanged. (\Cref{fig:Bad final 2}, C)
    
    \item [c.] If the starting point of the excursion is blue and the ending point is not in ${\sf f}$ then color the vertices added in the excursion blue. That is, $\cB_{\tau_{j+1}} = \cB_{\tau_j} \cup \{\cN_{\tau_{j+1}} \setminus \cN_{\tau_j}\}$. Define $\cB_k = \cB_{\tau_j}$ for all $\tau_j \le k <\tau_{j+1}$. (\Cref{fig:Bad final 2}, B)
\end{itemize}
   
\end{itemize}  
\item Keep the blue vertex set unchanged for the final excursion.
\end{itemize}
Roughly, if we hit ${\sf f}$ in an excursion (but not immediately) then we imagine `a very bad event' has occurred and color everything blue. Otherwise, if an excursion starts in blue, we color the excursion blue and otherwise we keep the blue vertices unchanged.

\begin{figure}[h]
    \centering
    \includegraphics[scale=1.1]{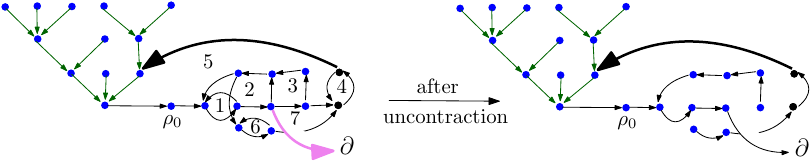}
    \caption{It shows that if in the last excursion, the process goes out from the set of blue vertices then after uncontraction, there will be an incoming edge to the set of vertices ${\sf F}_n$ sampled in stage 1.}
    \label{fig:Bad uncnt}
\end{figure}
For completion define $\cB_k = \cB_{J_{\max}}$ for all $k \ge \tau_{J_{\max}}$ and also $B_\infty = B_{\tau_{J_{\max}}}$, so for example $\cB_{\tau_j} =\cB_\infty= \cB_{\tau_{J_{\max}}}$ if $j >J_{\max}$.
Now we define a bad event
$$
\mathfrak B:=\text{the starting point of the final excursion is blue}.
$$
Finally, we define bad event for stage $3$ as follows: 
$$\mathfrak D:=\text{at least one of the local CLEB process steps hits $\sf f$ in stage 3.}$$
Recall that $M_n$ denoted the minimum spanning arborescence of $(\vec K_n, \partial)$ for weights $W_n$. We emphasize that the following lemma is deterministically true for any collection of generic weights. Define the future of $v$ in an arborescence $t$ to be the set of vertices in the oriented path from $v$ to the boundary (denoted $\mathfrak F_t(v)$). Define the past of $v$ to be the set of vertices $u$ such that $v \in \mathfrak F_t(u)$, and denote it by ${\sf  Past}_t(v)$ (see \Cref{fig:CLEB}).


\begin{prop}\label{lem:no_bad_no_disturb}
Suppose that $\cI\cI_n(\sf t )$ occurs in stage 1 and  $\mathfrak B \cup \mathfrak D$ does not occur in stages $2$ and $3$. Then all the edges in ${\sf t}$ belong to  $M_n$. Furthermore, there is no edge in $M_n\setminus{\sf t}$ whose head in ${\sf V}$ is in ${\sf f}$. 
\end{prop}
\begin{proof}
The proof proceeds in two steps. First we show that if $\mathfrak B$ does not occur, then after the uncontraction procedure, none of the edges of ${\sf t}$ is erased and no edge is added whose head is in ${\sf f}$. Recall from \Cref{lem:stage2_and_3flip} that we can first complete stage 2, uncontract all the cycles, then complete stage 3 and uncontract the cycles, and this does not affect the output of the local CLEB algorithm. We will exploit this here and first uncontract the cycles created in stage 2.

Suppose $\alpha_1 = \infty$, then the only vertices which are blue are in ${\sf f}$. In this case, no cycle is created with any edge in ${\sf f}$  as ${\sf f }$ is not hit.  So in this case the proposition is immediate.

 So now suppose $\alpha_1<\infty$. Suppose we  uncontract all the cycles created after some finite $\tau_j \ge \alpha_1$ (in reverse chronological order). After doing this, we obtain an arborescence $M_{\tau_j}$ of $G_{\tau_j}$ with boundary $\partial \cup \{{\sf V} \setminus \cN_\tauend\}$ (the vertices in ${\sf V} \setminus \cN_\tauend$ are singletons which are unexplored, see the general definition of arborescence in \Cref{def:general_arborescence}, and in effect $M_{\tau_j}$ is a spanning arborescence of the graph spanned by $\cN_{\tauend}$ with boundary $\partial$). See \Cref{fig:excursion} for reference.
 
We claim that if $\mathfrak B$ does not occur then the following is true for all times $\tau_j$ in $\hitpast \cap [\alpha_1,\infty)$.
\begin{itemize}
    \item ${\sf Past}_{M_{\tau_j}}(X_{\tau_j}) = {\sf t}_{\tau_j}$ 
    \item The tail in $\Tip_{\tau_j}$ of the outgoing edge of $M_{\tau_j}$ from $X_{\tau_j}$  is not blue. 
\end{itemize}
\begin{figure}[h]
    \centering
\includegraphics[scale=0.8]{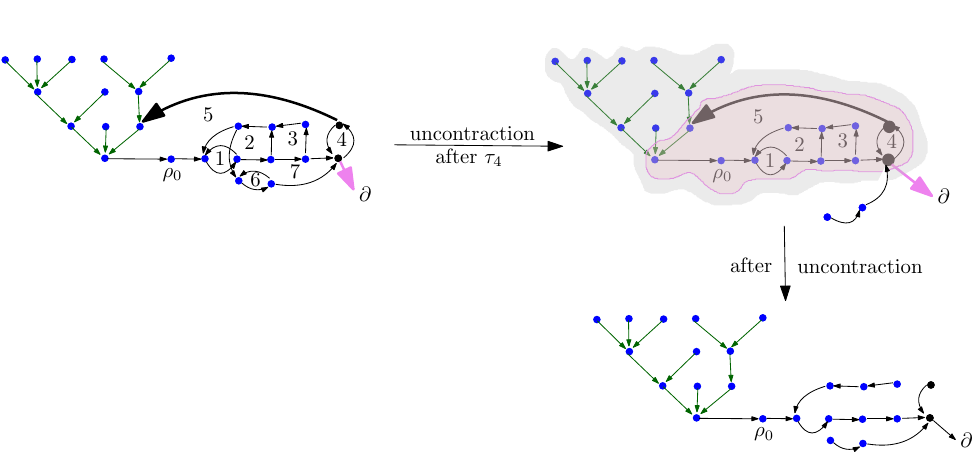}
    \caption{An illustration of how a good event in stage 2 does not `disturb' ${\sf t}$. The cycles are numbered in the order of their creation. The last cycle (numbered 7) is created at time $\tau_5 = \tau_{J_{\max}}$. Also $\alpha_1=\tau_4$ happens when cycle numbered 5 is created.
    The last excursion is marked in violet. Here the bad event doesn't occur, that is, the last excursion starts at a non-blue vertex. If we uncontract all the cycles created after $\tau_4=\alpha_1$, then the past of $X_{\tau_4}$ (which is the green edges in the gray part) is ${\sf t}_{\tau_4}$ (the green edges intersecting the gray region). Also notice that the last excursion (colored violet) starts at a non-blue vertex. After uncontracting cycle $5$, the thick oriented edge hitting ${\sf f}$ must be removed. After uncontracting all the cycles, the incoming tree isn't disturbed as claimed in \Cref{lem:no_bad_no_disturb}.}
    \label{fig:uncnt_good}
\end{figure}
We first explain why it is enough to prove this claim (see \Cref{fig:uncnt_good} for an illustration). Suppose the claim is true at $\alpha_1$. Recall that we must contract a cycle in the $\alpha_1$-th step (call it $C$) and color $\cN_{\alpha_1} \setminus \Tip_{\alpha_1-1}$ is blue. Therefore by the claim, the outgoing edge from $C$ in $M_{\alpha_1}$ has to be in $\Tip_{\alpha_1-1}$. Therefore, when we uncontract $C$, we erase the oriented edge exposed in step $\alpha_1$. Also observe that ${\sf t}_k \supseteq {\sf t}$  for all $k<\alpha_1$ (since the tree is not hit at all before $\alpha_1$). Combining these two observations, after uncontracting $C$, the arborescence $M_{\alpha_1-1}$ obtained has all the oriented edges of ${\sf t}$ and furthermore, $M_{\alpha_1-1}$ contains no oriented edge whose head is in ${\sf t}$. Since no cycle is created strictly before time $\alpha_1$ containing any edge from ${\sf t}$, their uncontraction has no further effect on ${\sf t}$. Hence the proposition follows. 

We are left to prove the claim.
We prove the claim by backward induction in time, we refer to \Cref{fig:case1,fig:case2,fig:case3} for illustrations of the cases. For $j = J_{\max}$, observe that no cycle is created after $J_{\max}$ which contains any edge in $S_{\tau_{J_{\max}}}$. Hence the first edge $\vec e$ exposed in the final excursion is present in $M_{\tau_{J_{\max}}}$.
Since we are in the complement of $\mathfrak B$, the tail of $\vec e$ cannot be blue. Furthermore, since $S_{\tau_{J_{\max}}}$ is not hit, ${\sf t}_{\tau_{J_{\max}}}$ is `undisturbed'. Therefore, the claim is true for the base case $j={J_{\max}}$.

If $\tau_{J_{\max}}=  \alpha_1$ then we are done so suppose that is not the case. So suppose the claim is true for some $\tau_j \in \hitpast \cap (\alpha_1,\infty)$. Now we uncontract in reverse order all the cycles created  in the excursion $(\tau_{j-1},\tau_j]$ and argue that the claim is true at time $\tau_{j-1}$. 
\begin{figure}[h]
    \centering
\includegraphics[scale=1]{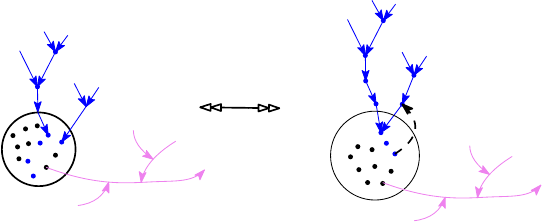}
    \caption{Proof of induction in \Cref{lem:no_bad_no_disturb}, $\tau_j =\tau_{j-1}+1$ case.}
    \label{fig:case1}
\end{figure}
Suppose first $\tau_{j-1} = \tau_j-1$, and this necessarily implies $\tau_j \in \hittree$ since $\tau_j > \alpha_1$ (that is, ${\sf f}$ is immediately hit after step $\tau_{j-1}$), see \Cref{fig:case1} for an illustration. In \Cref{fig:case1}, all the vertices in $\Tip_{\tau_{j}}$ are represented by the dots inside the disc in the left. We ignore here the graph structure and only keep track of the vertices which are blue. Similarly, $\Tip_{\tau_{j-1}}$ is represented by the dots inside the disc in the right hand side. 
Here, the only vertices left in $\cN_{\tau_{j-1}}$ which are not in $\Tip_{\tau_{j-1}}$ are in ${\sf f}$ as we assumed $\tau_{j-1} \ge \alpha_1$.
In this case, the cycle $C$ created thus has all vertices blue except $X_{\tau_{j-1}}$. By induction hypothesis, since the tail of the outgoing edge from $X_{\tau_{j}}$ in $M_{\tau_j}$ is not blue (the outgoing violet part in the left hand side of \Cref{fig:case1}), we must erase the edge exposed in step $\tau_j$ (the dashed edge in the right hand side of \Cref{fig:case1}). Therefore the outgoing edge of $M_{\tau_{j-1}}$ from $\Tip_{\tau_{j-1}}$ remains the same vertex, which is not blue by induction). Furthermore by induction, $\textsf{Past}_{M_{\tau_j}}(X_{\tau_j}) = {\sf t}_{\tau_j}$ and ${\sf t}_{\tau_j}$ is obtained from ${\sf t}_{\tau_{j-1}}$ by contracting the cycle $C$. Since we did not delete any edge in ${\sf t}$ when we uncontracted $C$, the claim is true at $\tau_{j-1}$ as well.

So now assume $\tau_{j}>\tau_{j-1}+1$. This excursion cannot end in ${\sf f}$ as this would make every vertex blue from $\tau_j$ onward, which violates the induction hypothesis. So now suppose the excursion ends in a vertex not in ${\sf f}$. This implies that no  new vertex of ${\sf f}$ is contracted into $X_{\tau_j}$ and hence ${\sf t}_{\tau_j} = {\sf t}_{\tau_{j-1}}$. This proves the first part of the claim. See \Cref{fig:case2}.

\begin{figure}
    \centering
\includegraphics[width=0.5\linewidth]{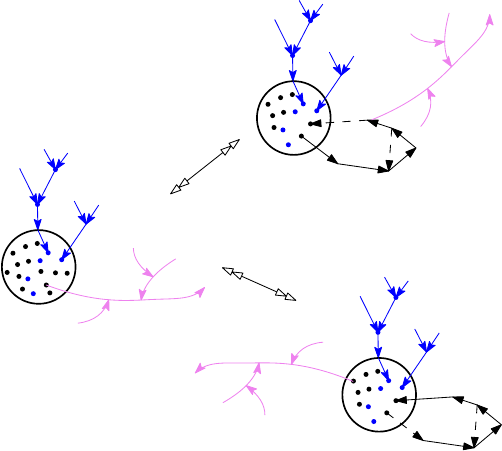}
    \caption{Proof of induction in \Cref{lem:no_bad_no_disturb}, $\tau_j > \tau_{j-1}+1$ case when the excursion starts at a non blue vertex.}
    \label{fig:case2}
\end{figure}

For the second part of the claim,
we now have two cases, the excursion starts either at a non-blue vertex or at a blue vertex.
Let us first consider the former case (see \Cref{fig:case2}), and recall that in this situation the coloring is unchanged from $\tau_{j-1}$ to $\tau_j$ by item b. of the construction of blue vertices.
Observe that the tail of the oriented edge of $M_{\tau_j}$ emanating from $X_{\tau_j}$ might or might not be a vertex in this excursion, but in either case, the second part of the claim is valid for $M_{\tau_{j-1}}$. Indeed, in this case, the edges we delete must be in the excursion and hence  cannot be in ${\sf t}$. Furthermore $M_{\tau_{j-1}}$ is obtained from $M_{\tau_j}$ by adding to it the undeleted edges in the excursion appropriately. Then the second part of the claim follows in the top right case of \Cref{fig:case2} because the excursion starts at a non blue vertex. For the bottom right case in \Cref{fig:case2}, the second item of the claim holds by induction hypothesis.

Now suppose the excursion starts at a blue vertex. 
By the induction hypothesis, the future of $X_{\tau_j} $ in the arborescence $M_{\tau_{j}}$ starts at a non blue vertex in $\Tip_{\tau_j}$. Therefore, in the first cycle uncontracted, the vertex with two outgoing edges is $X_{\tau_{j-1}}$.   Therefore, the first edge in this excursion is removed.
Thus the tail of the outgoing edge from $X_{\tau_{j-1}}$ of $M_{\tau_{j-1}}$ remains a non-blue vertex as desired. Thus the claim remains valid in this case also, see \Cref{fig:case3}.
\begin{figure}
    \centering
    \includegraphics[width=0.5\linewidth]{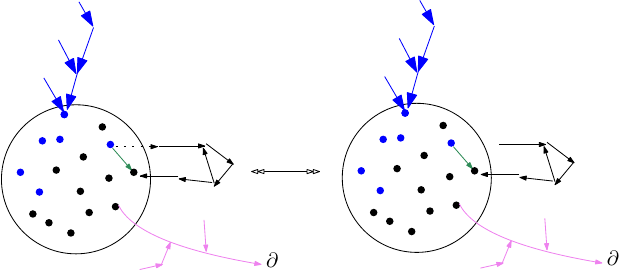}
    \caption{Left: In $X_{\tau_{j-1}}$, the blue vertex has the green outgoing edge. Then it goes out with the dotted arrow and form $X_{\tau_j}$. Right: After uncontraction, the dotted arrow is deleted.}
    \label{fig:case3}
\end{figure}

Now we move to stage 3 which is easier to handle.
    Observe that in stage 3, no cycle is  created with any edge in ${\sf t}$ and hence if $\mathfrak D$ does not occur, no edge is added whose head is in ${\sf f}$.
This completes the proof of the proposition.
\end{proof}
So the crucial proposition we need to prove is the following:
\begin{prop}\label{prop:bad_event_zero_prob}
$\P^{\sf t}(\mathfrak B \cup \mathfrak D) \to 0$ as $n \to \infty$.
\end{prop}
In \Cref{sec:bad_event_B,sec:bad_event_D}, we show separately that the probability of the events $\mathfrak B, \mathfrak D$ tends to 0. Assuming \Cref{prop:bad_event_zero_prob}, we can  now complete the proof of \Cref{thm:main}
\begin{proof}[Proof of \Cref{thm:main}]
From \Cref{prop:bad_event_zero_prob,prop:incoming_tree}, we also obtain that the unconditional probability $\P(\mathfrak B \cup \mathfrak D) \to 0$ as $n \to \infty$.
    Fix $R>0$. Observe that for any ${\sf t}$, 
    \begin{multline*}
    \P(B_{(({\sf V}_n,M_n),1)}(R) \simeq {\sf t}) =  \P(B_{(\Upsilon_{n},1)}(R)\simeq {\sf t} , \mathfrak B^c \cap \mathfrak D^c)+o(1)\\ = \P(B_{(\Upsilon_{n},1)}(R)\simeq {\sf t}) +o(1) = \P(B_{(\vec {\sf T},1)} \simeq {\sf t})+o(1),
    \end{multline*}
    where for the first and second equalities we use \Cref{prop:bad_event_zero_prob} and bounded convergence theorem. For the last equality we used \Cref{prop:incoming_tree}. This completes the proof.
\end{proof}
 \section{Bad events do not occur in stage 2}\label{sec:bad_event_B}

   In this section, our goal is to prove the following proposition. Recall the definition of $\Pt$ from \eqref{eq:Pt}.
   \begin{prop}\label{prop:bad_B}
      $\Pt(\mathfrak B) \to 0$ as $n \to \infty$.
   \end{prop}
   The idea is to bound the growth of blue vertices so that the probability of picking a blue vertex to start the final excursion is $o(1)$.
In the next subsection we will prove some preliminary estimates, which we shall subsequently use to establish a bound on the growth of blue vertices.

\subsection{Preliminary estimates}\label{sec:preliminary}

   Let $\cF_k$ be the sigma algebra generated by the local CLEB process at any step $k>\tauendone$. Equivalently, this is the sigma algebra generated by the sequence of exposed edges $(S_{\tauendone}+1,\ldots, S_k)$. We start with a lemma involving hitting probabilities which will be used repeatedly in what follows.
\begin{lemma}\label{lem:hitting_prob}
  Fix $k>\tauendone$ and let $A \subset V_k \setminus X_k$. Let $H_A$ be the event that the local CLEB process hits $\cont(A)$ in step $k+1$. Then 
  $$
  \Pt(H_A|\cF_k)=\frac{|\cont(A )\setminus {\sf f}|}{n-|{\sf f}|-1}\mathds 1_{\tip_k =1} + \frac{|\cont(A)|}{n-\tip_k}\mathds 1_{\tip_k>1}
  $$
 almost surely. 
\end{lemma}
\begin{proof}
    It's a simple consequence of the fact that in stages 2 and 3, the local CLEB process can't hit $\sf f$ in step $k+1$  if in step $k$ a previously unexposed vertex is hit, which occurs if and only if $\tip_k=1$.
    This is a consequence of parts b. and c. of \Cref{cor:cond_law}. 
    
    Now observe that the number of outgoing edges from $X_k$ which hits any set $A$ is $\tip_k |\cont(A)|$. The number of outgoing edges from $X_k$ whose head is not in $X_k$ is $ \tip_k (n-\tip_k)$. If $\tip_k=1$, the number of outgoing edges which hits $A$ but not $|{\sf f}|$ is $|\cont(A) \setminus {\sf f}|$ and the total number of outgoing edges which the CLEB walk can take is $n-1-|{\sf f}|$. Therefore,
\begin{align*}
\Pt(H_A|\cF_k)\mathds 1_{\tip_k=1}& =\frac{|\cont(A)\setminus {\sf f}|}{n-|{\sf f}|-1}\mathds 1_{\tip_k=1}, \\ \Pt(H_A|\cF_k)\mathds 1_{\tip_k>1} &=\frac{|\cont(A)|}{(n-\tip_k)}\mathds 1_{\tip_k>1},
\end{align*}
where both the equalities hold almost surely.
\end{proof}
For $j \ge 1$, recall that $\cN_j$ is the set of vertices exposed up to step $j$ and $N_j = |\cN_j|$.
We show that $N_j$ grows linearly in $j$ with very high probability. Most of our estimates are needed much before $O(n)$ steps are concluded, but we also need to reach a timescale which is quite close to $n$. It turns out that analyzing the process up to time $ n/(\log \log n)$ is enough for our purposes (though it is not optimal for the following estimates to hold, simply a convenient choice). To that end, define
\begin{equation}
    A_n = \log (\log (n)) \text{ and }\eps_n  = A_n^{-1} \label{eq:epsn}
\end{equation}
\begin{lemma}\label{lem:Nk_lower}
Take a sequence $\delta_n \to 0$. For every $k\le \delta_n n $, $N_{k } \le k +1$ and
$$
\Pt( N_k \le k/2, k \le \tauend) \le \left(\frac{2ek}{n} \right)^{k/2} \le (0.5)^{\delta_n n},
$$
for all large enough $n$.
\end{lemma}
\begin{proof}
    The upper bound is deterministic and trivial as in each step of the local CLEB process, at most one vertex is added. For $k \in (\tauendone, \tauend]$, let $\cA_k$ be the event such that in step $k$, the CLEB walk hits $\cN_{k-1}$. Also because of the conditioning, $N_{\tauendone}  = \tauendone +1 = |{\sf f}|+1$. Therefore
\begin{equation}
N_{k \wedge \tauend}=k\wedge \tauend-\sum_{j=\tauendone +1}^{k\wedge \tauend} \mathds 1_{\cA_j}. 
\label{eq:Bj1}
\end{equation}
Note that using \Cref{lem:hitting_prob},
$$\Pt(\cA_j|\cF_{j-1})\mathds 1_{j\le\tauend}\le\frac{N_{j-1}}{n-\tip_{j-1} - |{\sf f}|}\le\frac{j}{n-j-|{\sf f}|},$$
almost surely. Consider the sequence of independent random variables $(\xi_j)_{j\ge 1}$, such that $\xi_j\sim\text{Bernoulli}(j/(n-j - |{\sf f}|))$. Because of the upper bound in \eqref{eq:Bj1}, we can stochastically dominate $\sum_{j=\tauendone +1}^{k \wedge \tauend} \mathds 1_{\cA_j}$ by $\sum_{1 \le j \le k}\xi_j$. Observe that for all $k \le \delta_n n$ and $n $ large enough,
$$
\sum_{j=1}^k\mu_j \le \frac{k(k+1)}{n},
$$
where $\mu_j:=\E(\xi_j)=j/(n-j-|{\sf f}|)$, since $\delta_n \to 0$.
By Chernoff bound,
$$
\Pt\left(\sum_{1 \le j \le k}\xi_j \ge k/2\right) \le  \left(\frac{2e(k+1)}{n} \right)^{k/2} \le (4e\delta_n)^{\delta_n n/2},
$$
where the last inequality follows from the monotonicity in $k$ as long as $k $ is much smaller than $n$.
This completes the proof since $\delta_n \to 0$.
\end{proof}

Now we upper bound the total number of steps taken in the stage $2$ of the local CLEB process. The next two lemmas show that this quantity is highly concentrated around $n$.
\begin{lemma}\label{lem:total_steps}
For all $k \ge 0$,
$$
\Pt(\tauend -\tauendone > k) \le \left(1-\frac1n\right)^{k}.
$$
In particular, 
$$
\Et(\tauend -\tauendone) \le n.
$$   
\end{lemma}
\begin{proof}
Using \Cref{lem:hitting_prob},
$$\Pt(\text{the CLEB walk hits $\partial$ in $(k+1)$-th step} | \cF_k)\mathds 1_{\tauend > k}=\frac{1}{n-\tip_k}\geq\frac{1}{n},$$
for all $k\geq\tauendone$, almost surely. Thus we see that $(\tauend -\tauendone)$ is stochastically dominated by Geometric$(1/n)$. Also 
$$
\Et(\tauend -\tauendone ) \le \sum_{j \ge 0}  \left(1-\frac1n\right)^{j} = n,
$$
which completes the proof.
\end{proof}
We now prove a lower bound for $\tauend$.
\begin{lemma}\label{lem:tauend_lower}
For any sequence $\delta_n \to 0$, we have
$$
\Pt(\tauend-\tauendone\le \delta_n n) =O(\delta_n).
$$ 

\end{lemma}
\begin{proof}
    Using \Cref{lem:hitting_prob},  for every $k \le \delta_n n+\tauendone$, 
   $$
   \Pt(\text{hit $\partial$ in step $k+1$ of the CLEB walk}| \cF_k)\mathds1_{\tauend>k} \le \frac{1}{n-k-1} \le \frac{1}{n-\delta_n n - 1},
   $$
   almost surely (since, trivially, $\tip_k \le N_k \le k+1$). So by union bound, the probability that $\tauend - \tauendone \le \delta_n n$ is at most 
   $$
   \frac{\delta_n n}{n-\delta_n n - 1} \le 2\delta_n,
   $$
   for all large enough $n$.
\end{proof}

Now we show that it takes at least $O(n)$ steps to hit the incoming tree ${\sf f}$.  

\begin{lemma}\label{lem:alpha1}
We have
$$
    \Pt(\alpha_1 \le 2n\eps_n) =O(|{\sf f}|\eps_n), 
    $$    
    for all large enough $n$. 
\end{lemma}
 \begin{proof}
     Suppose $H_{\sf f}$  is the event that ${\sf f}$ is hit in step $k+1$.
   Using \Cref{lem:hitting_prob},  for all $\tauendone\le k<n\eps_n$
    $$
    \Pt(H_{\sf f}|\cF_{k}) \le  \frac{|\sf f|}{n-\tip_k} \le \frac{|\sf f|}{n-2n\eps_n},
    $$
    almost surely, since $\tip_k\le N_k\le k+1\le 2n\eps_n$.
    Therefore, by union bound,
    $$
    \Pt(\alpha_1 \le 2n\eps_n) \le \frac{2n\eps_n|{\sf f}|}{n-2n\eps_n}  \le 4|{\sf f}|\eps_n, 
    $$
    for all large enough choices of $n$.
    This completes the proof.
 \end{proof}

Next, we show that although by time $o(n)$, the chance to hit ${\sf f}$
 is small, the CLEB walk does hit vertices close to ${\sf f}$ with high probability. Recall $\eps_n^{-1} =A_n$.
\begin{lemma}\label{lem:hit_close_to_tree}
 For all $n \ge 1$,
 \begin{align*}
  \Pt(\text{a vertex in  $\cN_{ A_n^2}$ is hit in steps  $(n\eps_n,2n\eps_n)$})  \ge 1-O(|{\sf f}|\eps_n).
 \end{align*}
 \end{lemma}
\begin{proof}
By \Cref{lem:hitting_prob}, for any $k\in (n\eps_n,2n\eps_n)$, we know that
$$\Pt(\text{hit $\cN_{ A_n^2}$ in step $k$ of the CLEB walk$|\cF_{k-1}$)$\mathds1_{\tauend-\tauendone>k-1}$}\ge\frac{N_{ A_n^2}-|{\sf f}|}{n},$$
almost surely. Let us consider the following event, $$\cE_n:= \text{the CLEB walk does not hit $\cN_{ A_n^2}$ in any step in the interval $(n\eps_n,2n\eps_n)$.}$$
We obtain,
$$\Pt(\cE_n|N_{ A_n^2}\ge A_n^2 /2)
\le\left(1-\frac{ A_n^2}{4n}\right)^{n\eps_n}\le\exp(- A_n /4),$$
for all large enough $n$. Also by \Cref{lem:Nk_lower}, 
$$\Pt(N_{ A_n^2} \le A_n^2/2,\tauend-\tauendone>2n\eps_n)\le (0.5)^{n\eps_n},$$
for all large enough $n$. Now,
\begin{align*}
\Pt(\cE_n)&\le\Pt(\cE_n\cap\{\tauend-\tauendone>2n\eps_n\})+\Pt(\tauend-\tauendone\le 2n\eps_n)&\\
&\le\Pt(\cE_n\cap\{N_{ A_n^2}> A_n^2/2\}\cap\{\tauend-\tauendone> 2n\eps_n\})+(0.5)^{n\eps_n}+O(|{\sf f}|\eps_n)&\\
&\le\exp(- A_n /4)+(0.5)^{n\eps_n}+O(|{\sf f}|\eps_n)=O(|{\sf f}|\eps_n).&
\end{align*}
Taking the complement of $\cE_n$, we conclude.
\end{proof}
\noindent
Obviously, $\cN_{ A_n^2}$ is hit at a time in $\hitpast$. We now want to say that $\cN_{A_n^2}$ is hit at a time which is large enough, and by the time $\cN_{A_n^2}$ is hit for the first time, ${\sf f}$ is not hit with high probability.
\begin{equation}
   I := \min\{j : \tau_j>n\eps_n  \text{ and $\cN_{A_n^2}$ is hit in step $\tau_j$}\}, \label{eq:I}
\end{equation}
and define \begin{equation}
    \cS:= \left\{\tau_I < 2n\eps_n, N_{\tau_I} \ge \frac{n\eps_n}{2}\right\}. \label{eq:cS}
\end{equation}
\begin{lemma}\label{lem:cS}
    We have 
    $$
    \Pt(\cS) \ge 1-O(|{\sf f}|\eps_n). 
    $$
\end{lemma}
\begin{proof}
    This follows from \Cref{lem:hit_close_to_tree,lem:Nk_lower,lem:tauend_lower}.
\end{proof}
We now show that after $\tau_I$, it is unlikely for the intervals between two occurrences of $\hitpast$ to be very long. In fact these intervals are roughly Geometric $({N_{\tau_j}}/n)$  (the snowball effect as described in \Cref{sec:outline}).
\begin{lemma}\label{lem:interval_small}
For every $m\ge 1, j \ge 1$,
$$
\Pt(\tau_{j+1} - \tau_j> m| \cF_{\tau_j})\mathds1_{\tau_j<\infty} \le \left(1-\frac{N_{\tau_j}-|{\sf f}|}{n}\right)^m\mathds1_{\tau_j<\infty},
$$
almost surely.    
\end{lemma}
\begin{proof}

    Observe that conditioned on the $\sigma$-algebra $\cF_k$ and on the event $ \tau_j \le k <\tau_{j+1}$, the probability that $k+1=\tau_{j+1}$ is at least 
$$
\frac{N_{\tau_j}-|{\sf f}|}{n-\tip_k} \ge \frac{N_{\tau_j}-{|\sf f|}}{n}.
$$
Since this bound is independent of $k$, iterating this bound, we get
$$
\Pt(\tau_{j+1} - \tau_j \in (m,\infty)| \cF_{\tau_j})\mathds1_{\tau_j <\infty}  \le \left(1-\frac{N_{\tau_j}-{|\sf f|}}{n}\right)^m \mathds1_{\tau_j<\infty},
$$
almost surely, as desired.
\end{proof} 

\subsection{Bounding the growth of blue vertices}\label{sec:blue_bound}
Recall the inductive definition of blue vertices, and that we need to control the growth of blue vertices. We first show that item a. in the iterative definition of the growth of blue vertices do not happen with probability tending to 1.

Recall the definition of $\tau_I$ from \eqref{eq:I} (the first time $\cN_{A_n^2}$ is hit after time $n\eps_n$). Let $\alpha_1 = \tau_{j_1}$. Suppose for $j_1\le j \le J_{\max}$,
$$
\cV_j:=\{\exists i:  j_1 \le i \le j, \tau_{i-1}+1< \tau_{i}, \tau_i \in \hittree \}
$$
and for $j>J_{\max}$, $\cV_j = \cV_{J_{\max}}$. In words, $\cV_{J_{\max}}^c$  contains the event that all the times in $\hittree$ occur immediately after a time in $\hitpast$. Said otherwise,  $\cV_{J_{\max}}$ entails that the event in item a. of the iterative definition of blue vertices occurs in some excursion. We insist that as long as $\alpha_1 > \tau_I$,  $\cV_{J_{\max}}$ includes the event that when the CLEB walk hits ${\sf f}$ for the first time, it does not do so immediately after some time in $\hitpast$ (i.e. $\tau_{j_1-1} <\tau_{j_1}-1$). Note that if this event occurs, it will create a lot of blue vertices at time $\alpha_1 = \tau_{j_1}$ which is undesirable. Also observe that $\cV_j$ is increasing, and hence if $\cV_{J_{max}}$ happens, we have no hope of reasonably bounding the growth of blue vertices.  This is why we prove the next lemma which shows that $\Pt(\cV_{J_{\max}}) = o(1)$.

\begin{lemma}\label{lem:hk}
    $$
    \Pt(\cV_{J_{max}}) \to 0
    $$
    as $n\to \infty$.
\end{lemma}
\begin{proof}
We first recall the idea. We ensured in \Cref{cor:cond_law} that to hit ${\sf f}$, in step $k+1$, it must be the case that $\tip_k>1$. Thus if $\cV_{J_{\max}} \cap \{\alpha_1>\tau_I\}$ were to happen, at some point $k$ such that $\tau_I<\tau_j < k<\tauend$, the CLEB walk must hit a previously exposed vertex but $k \not \in \hitpast$, and then hit ${\sf f}$ in step $k+1$. Since the excursions are small (\Cref{lem:interval_small}), the expected number of times when such a $k$ occurs is  $o(n)$, which is not large enough to hit ${\sf f}$ with positive probability.

We now expand on this.
For every $k$, define $\tau_{j_k}$ to be the largest element in $\hitpast$ strictly smaller than $k$ (if there is no such element, define $\tau_{j_k} = \tauendone$).
Now define $\cH_k$ to be the event that $\tau_I<k-1 <\tauend$, $\cN_{k-1}\setminus \cN_{\tau_{j_k}} \neq \emptyset$ and
\begin{itemize}
    \item in step $k$, the CLEB walk hits $\cN_{k-1}\setminus \cN_{\tau_{j_k}}$,
    \item $k+1 \in \hittree$.
\end{itemize}
As a consequence of the fact that ${\sf f}$ cannot be hit from a brand new exposed vertex (\Cref{cor:cond_law}), we must have
$$
\cV_{J_{\max}}\subseteq (\cup_{k \ge 1}\cH_k) \cup\{\alpha_1\le \tau_I\}.
$$
Recall the event $\cS$ from \eqref{eq:cS}, which we mention here for convenience:
$$
\cS:= \left\{\tau_I < 2n\eps_n, N_{\tau_I} \ge \frac{n\eps_n}{2}\right\}.
$$
Using  \Cref{lem:cS,lem:alpha1}, we know that 
\begin{equation}
  \Pt(\tilde \cS) := \Pt(\cS \cap \{\alpha_1>\tau_I\}) \to 1. \label{eq:tildecs}  
\end{equation}
 Therefore it is enough to upper bound $$\cup_{k> \tau_I}\cH_k \cap\tilde \cS.$$ Observe that $\tilde \cS$ is $\cF_{\tau_I}$-measurable.
For $k>\tauendone+1$, let $\cU_{k-1} $ be the event 
\begin{itemize}
    \item $k-1 <\tauend$, 
    \item $0<N_{k-1}- N_{\tau_{j_k} } \le n^{\delta_0}$
\end{itemize}
for some fixed small $\delta_0>0$. In other words, $\cU_{k-1}$ is the event that the excursion in which the index $k$ resides is not too big.
Observe that $\cU_{k-1}$ is $\cF_{k-1}$ measurable. Thus we can write using union bound
\begin{align*}
    \Pt(\cup_{k > \tau_I}\cH_k \cap \tilde \cS) \le \sum_{k \ge 1}\Et(\Pt(\cH_k|\cF_{k-1})\mathds1_{\cU_{k-1}}\mathds1_{\tau_I <k \le \tauend}) + \Pt(\cup_{k >\tau_I}(\cH_k \cap \cU^c_{k-1} \cap \tilde \cS ) ).
\end{align*}
We separately bound each term. First, using \Cref{lem:hitting_prob}, 
\begin{equation}
    \Pt(\cH_k|\cF_{k-1})\mathds1_{\cU_{k-1}} \le \frac{n^{\delta_0 }}{n-n^{\delta_0} - |{\sf f}|}\frac{|{\sf f}|}{n-n^{\delta_0}  - |{\sf f}|} \mathds1_{\cU_{k-1}}\le O\left(n^{\delta_0-2}\right)\mathds1_{k-1<\tauend},
\end{equation}
almost surely. The constant in $O$ above depends on $\delta_0$ and ${|{\sf f}|}$ and nothing else. Summing over $k$,
\begin{equation}
  \sum_{k \ge 1}\Et(\Pt(\cH_k|\cF_{k-1})\mathds1_{\cU_{k-1}}\mathds1_{\tau_I <k \le \tauend})\le O\left( n^{\delta_0-2}\right) \Et\left(\sum_{k \ge 1}\mathds1_{k-1<\tauend}\right) =O\left(n^{\delta_0-1}\right),\label{eq:bad1}
\end{equation}
where we used \Cref{lem:total_steps} to bound $\Et(\tauend)$.
Next, for the second term, we simply bound the probability of the union over $k > \tau_I$ of $\cU_{k-1}^c  \cap \tilde \cS$. Following the definition of $\cU_k$:
$$
\Pt(\cup_{k > \tau_I}\cU^c_{k-1} \cap \tilde{\cS}) \le \Pt(\cup_{j \ge I} \{\tau_j <\infty,\tau_{j+1} -\tau_j \in (n^{\delta_0},\infty)\} \cap \tilde \cS).
$$
Now note using \Cref{lem:interval_small},
\begin{multline*}
\Pt(\tau_{j+1}-\tau_j \in (n^{\delta_0},\infty) |\cF_{\tau_j})\mathds1_{\tau_I\le \tau_j <\infty}\mathds1_{\tilde \cS}\\ \le \left(1-\frac{n\eps_n - 2|{\sf f}|}{2n}\right)^{n^{\delta_0}}\mathds1_{\tilde \cS , \tau_I\le \tau_j<\infty}\le \exp\left(-O\left(\eps_n n^{\delta_0}\right)\right)\mathds1_{\tilde \cS , \tau_I\le \tau_j<\infty},
\end{multline*}
almost surely. We used the fact that if $\tau_j \ge \tau_I$, then on the event $\cS$, $N_{\tau_j} \ge N_{n\eps_n} \ge n\eps_n/2$.
Summing over $j$ and using union bound,
\begin{multline*}
\Pt(\cup_{j \ge I} \{\tau_j <\infty,\tau_{j+1} -\tau_j \in (n^{\delta_0},\infty)\} \cap \tilde{\cS}) 
\le \exp\left(-O\left(\eps_n n^{\delta_0}\right)\right) \Et(\tauend)\\ \le n\exp\left(-O\left(\eps_n n^{\delta_0}\right)\right).
\end{multline*}

The right hand side above tends to 0 because of our choice of $A_n=\eps_n^{-1} = \log \log (n)$. Combining the above estimates, the proof of the lemma is complete.
\end{proof}

Next, we need to control the growth of blue vertices on the complement of $\cV_{J_{\max}}$. Recall that $\cB_k$ denotes the set of blue vertices after step $k$. We first make a simple observation. After stage $1$ blue vertices start growing at $\alpha_1$ and also recall the definition of $I, \cS$ and $\tilde \cS$ from \eqref{eq:I}, \eqref{eq:cS}, \eqref{eq:tildecs} and that $\alpha_1 = \tau_{j_1}$. Observe that
\begin{equation}
     \cS':=\tilde \cS \cap\cV^c_{j_1} = \cS \cap \{\tau_{j_1} > \tau_I\} \cap \cV^c_{j_1} \subseteq \{\cB_{\alpha_1} \subseteq \cN_{A_n^{2}}\}. \label{eq:Balpha_small}
\end{equation}
Indeed, at time $\tau_I$, the number of vertices not in the tip is at most $N_{A_n^{2}}$ and since $\cV_{j_1}$ does not occur, $\tau_{j_1}-1 \in \hitpast$. Hence the number of vertices not in the tip at time $\tau_{j_1}-1$ is contained in $\cN_{A_n^{2}}$. Consequently, number of blue vertices at time $\alpha_1 = \tau_{j_1}$ must be at most $N_{A_n^{2}} \le A_n^{2}$.
Note that $\cS'$ is in the $\sigma$-algebra $\cF_{\alpha_1}$.

Our next lemma tells us how the set of blue vertices grow from excursion to excursion. Let $\Blue_j$ be the event that in step $\tau_j+1$, the oriented edge selected has tail blue. Said otherwise, the bad event $\mathfrak B$ is simply $\Blue_{J_{\max}}$.

\begin{lemma}\label{Bluej}
Conditioned on $\cF_{\tau_j}$ such that $\tau_j<\infty$, the distribution of the tail of the oriented edge selected in the next step is uniform over $\tip_{\tau_j}$. Consequently,
$$\Pt(\Blue_j|\cF_{\tau_j})\mathds1_{\tau_j<\infty} = \frac{|\cB_{\tau_j} \cap \Tip_{\tau_j}|} {\tip_{\tau_j}} \le \min\left\{\frac{B_{\tau_j}}{\tip_{\tau_j}},1\right\},$$
almost surely.
\end{lemma}
\begin{proof}
    Observe that in step $\tau_j+1$, we select an edge uniformly among the set of edges whose tail is in $\cont(X_{\tau_j})  = \tip_{\tau_j}$ and head is outside $X_{\tau_j}$ (since by definition, $\tip_{\tau_j} >1$). The number of such edges is the same for every vertex in $\Tip_{\tau_j}$ since we are in the complete graph. The lemma follows. 
\end{proof}
We need the $\min$ in \Cref{Bluej} as if $\cV_{J_{\max}}$ occurs, then it could be the case that every vertex exposed is blue. Our next lemma gives a recursion controlling the growth of blue vertices.
\begin{lemma}\label{lem:recursion_B}
    The following holds almost surely 
    $$
    B_{\tau_{j+1}}\mathds1_{\tau_{j+1}<\infty, \cV_{j+1}^c} =B_{\tau_j}\mathds1_{ \tau_{j+1}<\infty, \cV_{j+1}^c} +( \tau_{j+1} - \tau_j)\mathds1_{\Blue_j,\tau_{j+1} < \infty, \cV_{j+1}^c}.
    $$
\end{lemma}
\begin{proof}
    This follows from the definition of the blue vertices. Indeed if both $\tau_{j+1}<\infty$ and $\cV_{j+1}^c$ occurs, then the only way for blue vertices to change is if the excursion started from a blue vertex, and it changes by exactly the length of the excursion, which is $\tau_{j+1}-\tau_j$. The lemma follows.
\end{proof}
\noindent
We wish to use the recursion of \Cref{lem:recursion_B} to prove an upper bound of $$B_{\tau_{J_{\max}}}\mathds1_{\cV^c_{J_{\max}}}.$$ 
To effectively deal with the stopping times, we introduce an auxiliary sequence $\tilde B_{j}$ defined as follows which will dominate $B_{\tau_j}\mathds1_{\cV^c_j}$. Recall $\tau_{j_1} = \alpha_1$ and define $\tilde B_{j_1} = B_{\tau_{j_1}}\mathds1_{\cV^c_{j_1}}$. Define $N_{\infty}  = N_{\tau_{J_{\max}}}$ and note that because of the way we set up our notations, $N_{\tau_j}= N_\infty=N_{\tau_{J_{\max}}}$ for all $j>J_{\max}$. Now define for $j \ge j_1$,
\begin{equation}
    \tilde B_{j+1} =\tilde B_j +  {\sf Ge}_{j+1}\mathds1_{\widetilde{\Blue}_j}, \label{eq:tildeB_recursion} 
\end{equation}
where 
${\sf Ge}_{j+1}$ is distributed as a Geometric random variable with parameter $(N_{\tau_j}-{|{\sf f}|})/n$
independent of everything else and $\mathds1_{\widetilde{\Blue}_j}$ is Bernoulli with parameter $\min(\tilde B_{j}/\tip_{\tau_j},1)$ independent from everything else. Observe that, as a matter of convenience, we allow $\tilde B_j$ to grow, even if $j >J_{\max}$.

Let $S_{\le k}:=(S_i)_{0 \le i \le k}$ denote the sequence of exposed edges in the local CLEB process up to step $k$ and recall that $\cF_k$ is simply the $\sigma$-algebra generated by $S_{\le k}$.
\begin{lemma}\label{lem:B_domination}
 $\tilde B_j$ stochastically dominates $B_{\tau_j}\mathds1_{\cV_j^c}$ for all $j \ge j_1$. 
\end{lemma}
\begin{proof}
We prove this by induction on $j$.
Let $U_1,U_2,\ldots$ be a sequence of independent Unif$[0,1]$ random variables independent of everything else. Let $U_{\le j}$ denote the sequence $(U_1,\dots, U_j)$.
    Suppose we want to construct a measurable function $\varphi_j$ such that  $$\tilde B_j = \varphi_j(S_{\le \tau_j},U_{\le j}) \ge B_{\tau_j}\mathds1_{\cV^c_{\tau_j}},$$ almost surely where the first equality is in distribution. The base case $j=j_1$ is trivial as $\tilde B_{j_1} = B_{\tau_{j_1}}\mathds1_{\cV^c_{j_1}}$. Now suppose for some $j \ge j_1$, we have constructed $\varphi_j$. We consider two situations separately: $j\ge J_{\max}$ or $j < J_{\max}$. If $j\ge J_{\max}$,  we use $U_{\tau_{j+1}}$ to simply construct $\mathds1_{\widetilde{\Blue}_{j}}$
 and ${\sf Ge}_{j+1}$ independent of everything else, which defines $\varphi_{j+1}$ (this is a standard construction).   Observe that this ensures $\varphi_{j+1}(S_{\le\tau_{j+1}}, U_{\le j+1}) = \tilde B_{j+1}$ in distribution. Thus on the event $j \ge J_{\max}$,
    $$
    \varphi_{j+1}(S_{\le\tau_{j+1}}, U_{\le j+1})  \ge \varphi_j(S_{\le \tau_j},U_{\le j}) \ge B_{\tau_{J_{\max}}}\mathds1_{\cV^c_{J_{\max}}}  =B_{\tau_{j+1}}\mathds1_{\cV^c_{j+1}},
    $$
    almost surely. The first inequality is trivial as we have added some non-zero random variable and the second inequality is our induction hypothesis. The last equality follows from the convention that $B_{\tau_j} = B_{\tau_{J_{\max}}}$ and $\cV_j^c = \cV^c_{J_{\max}}$ for all $j \ge J_{\max}$.
    Now suppose we are on the event $j<J_{\max}$ 
    and now we use the expression on the right hand side of the recursion in \Cref{lem:recursion_B}. Note
    $$
    B_{\tau_j}\mathds1_{\tau_{j+1}<\infty}\mathds1_{\cV_{j+1}^c} \le B_{\tau_j}\mathds1_{\tau_j<\infty}\mathds1_{\cV_{j}^c} \le \varphi_j(S_{\le \tau_j},U_{\le j}),
    $$
    almost surely by our induction hypothesis and since $\cV_j$ is decreasing. Conditioned on $\cF_{\tau_j}$, $\mathds1_{\Blue_j}$ is stochastically dominated by $\mathds1_{\widetilde{\Blue}_j}$ since $B_{\tau_j}\mathds1_{\cV_j^c}$ is stochastically dominated by $\tilde B_j $. Thus we can construct a function $\tilde \varphi_{j+1}(S_{\le \tau_{j+1}}, U_{\le j+1})$ so that 
      $$
      ( {\sf Ge}_{j+1})\mathds1_{\widetilde{\Blue}_j} = \tilde \varphi_{j+1}(S_{\le \tau_{j+1}}, U_{\le j+1})  \ge ( \tau_{j+1} - \tau_j)\mathds1_{\Blue_j,\tau_{j+1} < \infty, \cV_{j+1}^c},
      $$
     almost surely, where the first equality is in distribution. Now define 
     $$
     \varphi_{j+1}(S_{\le \tau_{j+1}}, U_{\le j+1}) = \varphi_j(S_{\le \tau_j},U_{\le j}) + \tilde \varphi_{j+1}(S_{\le \tau_{j+1}}, U_{\le j+1}).
     $$
     This completes the proof of the coupling and the lemma.
\end{proof}
\begin{lemma}\label{lem:tilde_recursion}
We have 
$$\Et(\tilde B_j|\cF_{\alpha_1})\le\tilde{B}_{j_1}\left(1+\frac{n}{(N_{\alpha_1}-|{\sf f}|)^2}\right)^{j-j_1},$$
almost surely, for all $j>j_1$.
\end{lemma}
\begin{proof}
This follows essentially by iterating \Cref{eq:tildeB_recursion}.  Indeed, for every $\tau_j>\tau_{j_1}=\alpha_1$, all the vertices in $\Tip_{\alpha_1}$ are swallowed into $\Tip_{\tau_j}$. Therefore,
\begin{multline*}
\Et(\tilde{B}_j|\cF_{\tau_{j-1}})\mathds1_{j >j_1}\le\left(\tilde{B}_{j-1}+\frac{n}{N_{\tau_{j-1}}-|{\sf f}|}\frac{\tilde{B}_{j-1}}{\tip_{\tau_{j-1}}}\right)\mathds1_{j >j_1}\le \left(\tilde{B}_{j-1}+\frac{n}{N_{\alpha_1}-|{\sf f}|}\frac{\tilde{B}_{j-1}}{\tip_{\alpha_1}}\right)\mathds1_{j >j_1}\\\le \left(\tilde{B}_{j-1}+\frac{n}{N_{\alpha_1}-|{\sf f}|}\frac{\tilde{B}_{j-1}}{N_{\alpha_1}-|{\sf f}|}\right)\mathds1_{j >j_1},
\end{multline*}
almost surely, 
where for the first inequality we used \Cref{lem:interval_small}, the third inequality is due to the fact that $N_{\alpha_1}-|{\sf f}|\le\tip_{\alpha_1}$. Indeed, at time $\alpha_1$, all the vertices exposed between $\tauendone$ and $\alpha_1 $ is swallowed into the tip, except maybe some vertices of ${\sf f}$.
Taking expectation on both sides conditioned on $\cF_{\alpha_1}$, using the tower property, and iterating, we have our result.
\end{proof}
We now bound the number of blue vertices. Recall $\log(\log n)=A_n = \eps_n^{-1}$.
\begin{lemma}\label{lem:bound_blue}
    We have 
    $$
\Pt(B_{\tau_{J_{\max}}}\mathds1_{\cV^c_{J_{\max}}} >\exp(A_n^4)) \to 0.
    $$ 
\end{lemma}
\begin{proof}
    Note by \Cref{lem:total_steps}, 
    $$
    \Pt(\tauend-\tauendone > n A_n) \le e^{-A_n}.
    $$
    Also recall the event $\cS'$ from \eqref{eq:Balpha_small}. Recall 
    $$
    \Pt((\cS')^c ) =O(|{\sf f}|\eps_n),
    $$
    using \Cref{lem:cS,lem:alpha1}. By the stochastic domination proved in \Cref{lem:B_domination}, for any $A>0$,
    $$   \Pt(B_{\tau_{J_{\max}}}\mathds1_{\cV^c_{J_{\max}}} >A, \tauend-\tauendone \le  nA_n) \le \Pt(\tilde B_{nA_n} >A ) \le \Pt(\tilde B_{nA_n}\mathds1_{\cS'} >A) +\Pt((\cS')^c),
    $$
where for the first inequality we used the trivial inequality $J_{\max} \le \tau_{J_{\max}} \le \tauend$.
    By Markov's inequality, 
    \begin{equation}
        \Pt(\tilde B_{nA_n}\mathds1_{\cS'} >A ) \le A^{-1}\Et(\tilde B_{nA_n}\mathds1_{\cS'}).
    \end{equation}
    Now recall that $\cS'$ is $\cF_{\alpha_1}$-measurable. 
By \Cref{lem:tilde_recursion}, 
\begin{equation*}
    \Et(\tilde B_{nA_n}\mathds1_{\cS'}) =\Et(\Et(\tilde B_{nA_n}|\cF_{\alpha_1})\mathds1_{\cS'}) \le \Et\left( \tilde B_{j_1}\left(1+\frac{n}{(N_{\alpha_1}- |{\sf f}|)^2}\right)^{nA_n}\mathds1_{\cS'}\right). 
\end{equation*}
Now recall from \eqref{eq:Balpha_small}
$$
\cS' \subseteq \{\cB_{\alpha_1} \subseteq \cN_{A_n^2}\}\text{ and }\cS' \subseteq \left\{N_{\alpha_1} \ge \frac{n \eps_n}{2}\right\}.
$$
Since $\tilde B_{j_1} \le B_{\alpha_1}$, see that 
\begin{equation*}
  \Et(\tilde B_{nA_n}\mathds1_{\cS'}) \le A_n^{2}\left(1+O\left(\frac{1}{n\eps_n^2}\right)\right)^{nA_n} \le A_n^{2}\exp(O(A^{3}_n)) 
\end{equation*}
Plugging it back in,
\begin{equation*}
    \Pt(\tilde B_{nA_n}\mathds1_{\cS'} >A ) \le A^{-1}A_n^{2}\exp(O(A^{3}_n)).
\end{equation*}
Plugging $A = \exp(A_n^{4})$, we are done.
\end{proof}
Thus we have the corollary
\begin{corollary}\label{cor:blue_bound}
We have 
$$
\Pt(B_{\tau_{J_{\max}}} > \exp(A_n^{4}) ) \to 0.
$$   
\end{corollary}
\begin{proof}
    This follows from \Cref{lem:hk,lem:bound_blue}.
\end{proof}
\noindent
Recalling that $A_n = \log \log(n)$, we conclude that the number of blue vertices is at most $\exp(A_n^4)$ with high probability when we start the final excursion.

For any vertex $v$ in the tip, let $\cK_v = \cK_{v}(k)$ be the event that in step $k$ we select an oriented edge whose tail is $v$. Let $S_{\ge k}$ denote the full trajectory $(S_i)_{k \le i \le \tauend}$ after step $k$. Let $\cF_k'$ be the sigma algebra generated by $\cF_k$ and the selection of the tail of the oriented edge out of $\Tip_k$ which is to be selected.
\begin{lemma}\label{lem:future_couple}
For any $k \ge 1$, and $u,v \in \Tip_k$, the distribution of $S_{\ge k+2}$ conditioned on $\cF_k'$ on the event $\cK_v(k+1)$ is the same as that of $\cS_{\ge k+2}$ on the event $\cK_u(k+1)$. 
\end{lemma}
\begin{proof}
Suppose $\tip_k>1$ otherwise there is nothing to prove.
    Conditioned on $\cK_u(k+1)$, the edge selected in step $k+1$ is uniform over all edges outgoing from $u$ and the same is true for $v$. Since we are in the complete graph, for every $z \not \in \Tip_k$, there is an oriented edge whose head is outside $\Tip_k$ and tail is $u$ and the same is true for $v$. Thus the number of outgoing edges from $v$ we can choose is the same as that of $u$. Therefore, respectively for the conditional law conditioned on $\cK_u(k+1)$ and $\cK_v(k+1)$, we can couple this choice so that the head of the selected  edge is the same for both $u$ and $v$. After this we can couple the whole trajectories $S_{\ge k+2}$ to be exactly equal. Note that in every step of this coupling the set $(\cN_{j})_{j \ge 1}$ is exactly the same for both trajectories, so there is no issue with creating a cycle in one but not in the other.
\end{proof}

\begin{proof}[Proof of \Cref{prop:bad_B}]
Recall that $\mathfrak B  $ is the same event as $\Blue_{J_{\max}}$ (the last excursion starts from a blue vertex). Therefore it is enough to show that $\Pt(\Blue_{J_{\max}}) \to 0$.

    Recall that $S_{\le k}:=(S_i)_{0 \le i \le k}$ denoted the sequence of exposed edges in the local CLEB process. Observe that for any $u \in \Tip_{\tau_j}$ 
    \begin{equation*}    \Pt(\cK_u(\tau_j+1)|j=J_{\max},S_{\le \tau_j}) = \frac{\Pt(j=J_{\max}|\cK_u(\tau_j+1), S_{\le \tau_j} )\Pt(\cK_u(\tau_j+1)|S_{\le \tau_j})\mathds1_{\tau_j <\infty}}{\Pt(j=J_{\max}|S_{\le \tau_j})\mathds1_{\tau_j <\infty}}.
    \end{equation*}
    Using \Cref{lem:future_couple}, the first term in the numerator above does not depend on $u$ and by \Cref{Bluej}, the second term in the numerator is also independent of $u$. Therefore, conditionally on $S_{\le \tau_j}$ and the event $\{j = J_{\max}\}$, the tail of the next oriented edge selected is uniform over $\Tip_{\tau_j}$. Consequently,

    $$
   \Pt(\Blue_{J_{\max}}) = \Et(\Pt(\Blue_j|j=J_{\max},S_{\le \tau_j}))   \le\Et\left(\min\left(\frac{B_{\tau_{J_{\max}}}}{\tip_{\tau_{J_{\max}}}},1\right)\right).
    $$
Now observe that on the event $\alpha_1=\infty$, there are never any blue vertices in the tip, so the above expectation is 0. On the other hand, if $$\cG:=\{\alpha_1<\infty, N_{\alpha_1} \ge n\eps_n,B_{\tau_{J_{\max}}}\le\exp(A_n^{4})  \},$$ happens we have 
$$
\tip_{\tau_{J_{\max}}} \ge N_{\tau_{J_{\max}}} - |{\sf f}| \ge n\eps_n - |{\sf f}|.
$$
Thus, 
$$
\Et\left(\min\left(\frac{B_{\tau_{J_{\max}}}}{\tip_{\tau_{J_{\max}}}},1\right)\mathds1_{\cG}\right) =O\left(\frac{\exp(A_n^{4})}{n\eps_n}\right) =O\left(\frac{\exp((\log \log n)^4 )\log \log n}{n}\right) \to 0
$$
as $n \to \infty$.  Using \Cref{lem:cS,cor:blue_bound}, $\Pt(\cG) \to1$ thereby concluding the proof.
\end{proof}

\section{Bad events do not occur in stage 3}\label{sec:bad_event_D}
In this section we prove 
\begin{prop}\label{prop:D_to_zero}
We have 
$$
\Pt(\mathfrak D) \to 0
$$
as $n \to \infty$.    
\end{prop}

\begin{lemma}\label{lem:Ntauend}
For all $n\ge 1$,
    $$
    \Pt(N_\tauend >n^{0.9}) \ge 1-O(n^{-0.05}).
    $$
\end{lemma}
\begin{proof}
   Using  \Cref{lem:Nk_lower} and \Cref{lem:tauend_lower}  with $\delta_n = n^{-0.05}$,
   \begin{equation*}
       \Pt(N_{\tauend} \le n^{0.9}) \le \Pt(\tauend-\tauendone \le n^{0.95}) + \Pt(\tauend-\tauendone>n^{0.95}, N_{\tauend} \le n^{0.9}). 
   \end{equation*}
   The second term above has stretched exponential bound from \Cref{lem:Nk_lower}, which allows us to conclude.
\end{proof}

For every $x \in {\sf V } \setminus \cN_{\tauend}$, we define $\mathfrak D_x$ to be the event that the CLEB walk started at $x$ hits ${\sf f}$. Observe that $$\mathfrak D = \cup_{x \in {\sf V } \setminus \cN_{\tauend}} \mathfrak D_x.$$ For every $x \in {\sf V}\setminus \cN_{\tauend}$, let $\eta_x$ be the number of steps taken by the CLEB walk started at $x$ until it stops. We first prove an upper bound on $$\eta_{\text{max}} = \max_{x \in {\sf V} \setminus \cN_{\tauend}}\eta_x.$$ We will make use of the conditional probability measure $\P^{\tauend}.$  Suppose we have run stage $3$ up to time $\zeta_x$ and exposed a set of vertices $\cN_{\zeta_x}$ and we start the CLEB walk from $x$ at time $\zeta_x$.
\begin{lemma}\label{lem:etamax}
There exists a constant $C>0$ such that for all $n \ge 1$, $m \ge 1$, $x \in {\sf V} \setminus \cN_{\tauend}$,
$$
\Pt(\eta_{x} >m | \cF_{\tauend})\mathds1_{N_{\tauend} > n^{0.9}} \le \exp(-Cmn^{-0.1}),
$$
almost surely.
\end{lemma}
\begin{proof}
Conditioned on the $\sigma$-algebra up to stage $k >\zeta_x$, and assuming the walk hasn't stopped by step $k$ and $N_{\tauend}>n^{0.9}$, the probability of stopping the CLEB walk in step $k+1$ is at least (by using \Cref{lem:hitting_prob}) 
$$
 \frac{N_{\zeta_x} - |{\sf f}|}{n} \ge \frac{N_{\tauend}-|{\sf f}|}{n} \ge \frac{C}{n^{0.1}}.
$$
for some $C>0$ depending on ${|{\sf f}|}$ and nothing else.
So the probability that $\eta_x $ is at least $m$ is at most 
$$
\left(1-\frac{C}{n^{0.1}}\right)^{m} \le \exp(-Cm/n^{0.1}).
$$
\end{proof}
For any $\eta_x>k >\zeta_x$, define $\cJ_k$ to be the event that the CLEB walk hits 
\begin{itemize}
    \item $\cN_k \setminus \cN_{\zeta_x}$ in step $k+1$ and
    \item {\sf f} in step $k+2$.
\end{itemize}
Define $$\cJ_x:=\cup_{\zeta_x <k\le \zeta_x+\eta_x}\cJ_k.$$
We observe the following fact. 
\begin{lemma}\label{lem:DandJ}
    We have 
    $$
    \mathfrak D \subseteq \cup_{x \in {\sf V}\setminus  \cN_{\tauend}}\cJ_x.
    $$
\end{lemma}

\begin{proof}
    This follows from the description of local CLEB process as ${\sf f}$ cannot be hit if the tip is a single vertex (\Cref{cor:cond_law}).
\end{proof}


\begin{lemma}\label{lem:Jx}
    Fix $ x \in {\sf V} \setminus \cN_{\tauend}$. We have 
    $$
    \P^t(\cJ_x |\cF_{\tauend})\mathds1_{N_{\tauend} > n^{0.9}} \le \frac{C'}{n^{1.4}},
    $$
almost surely, for some $C'>0$.
\end{lemma}
\begin{proof}
We can write
 $$
 \P^t(\cJ_x |\cF_{\tauend})\mathds1_{N_{\tauend} > n^{0.9}} =\Et\left(\sum_{k \ge 1}\Et(\mathds1_{\cJ_k} | \cF_k)\mathds1_{\zeta_x < k \le \eta_x+\zeta_x} |\cF_{\tauend} \right)\mathds1_{N_{\tauend} > n^{0.9}}.
 $$
 Now observe that by \Cref{lem:hitting_prob},
 $$
 \Et(\mathds1_{\cJ_{k}} | \cF_k)\mathds1_{\zeta_x < k \le \eta_x+\zeta_x}  \le \frac{\eta_x}{n-\eta_x}\frac{|{\sf f}|}{n-\eta_x}\mathds1_{\zeta_x < k \le \eta_x+\zeta_x}, 
 $$
 almost surely. Indeed, at any step $k+1$,  the probability of hitting $\cN_k \setminus \cN_{\zeta_x}$, conditioned on $\cF_k$ is at most 
 $$
 \frac{|\cN_k \setminus \cN_{\zeta_x}|}{n-\tip_k} \le \frac{\eta_x}{n-\eta_x},
 $$
 since $\tip_k \le \eta_x$ and $|\cN_k \setminus \cN_{\zeta_x}| \le \eta_x$. Finally, the probability of hitting ${\sf f}$ at $k+2$ is at most $$\frac{|{\sf f}|}{n-\tip_k} \le \frac{|{\sf f}|}{n-\eta_x}$$
 for the same reason. Therefore,
 $$
 \sum_{k \ge 1}\Et(\mathds1_{\cJ_k} | \cF_k)\mathds1_{\zeta_x < k \le \eta_x+\zeta_x}  \le \frac{\eta_x^2 |{\sf f}|}{(n-\eta_x)^2},
 $$
 almost surely. Plugging this back in,
 \begin{equation*}
     \P^t(\cJ_x |\cF_{\tauend})\mathds1_{N_{\tauend} > n^{0.9}} \le \Et\left(\frac{\eta_x^2 |{\sf f}|}{(n-\eta_x)^2} \middle|\cF_{\tauend} \right)\mathds1_{N_{\tauend} > n^{0.9}},
 \end{equation*}
almost surely. Using the tail bound of $\eta_x$ obtained in \Cref{lem:etamax}, we see that 
 $$
 \Et\left(\frac{\eta_x^2 |{\sf f}|}{(n-\eta_x)^2} \middle|\cF_{\tauend} \right)\mathds1_{N_{\tauend} > n^{0.9}} \le \frac{n^{0.4}|{\sf f}|}{(n-n^{0.2})^2}n^{0.2}  +\sum_{m \ge n^{0.2}} m^2\exp\left(-\frac{Cm}{n^{0.1}}\right) \le \frac{C'}{n^{1.4}},
 $$
almost surely. We used the fact that $x \mapsto \frac{x^2}{(n-x)^2}$ is increasing in $x$ as long as $x <n$.
\end{proof}
\begin{proof}[Proof of \Cref{prop:D_to_zero}]
    Using \Cref{lem:Ntauend,lem:DandJ,lem:Jx}, we have 
    \begin{align*}
    \Pt(\mathfrak D) &\le  \Pt(\cup_{x \in {\sf V} \setminus \cN_{\tauend}}\cJ_x)\\
    & \le  \Pt(\cup_{x \in {\sf V} \setminus \cN_{\tauend}}\cJ_x|\cF_\tauend)\mathds1_{N_{\tauend}>n^{0.9}} + \Pt(N_{\tauend}\le n^{0.9})\\
    & \le \frac{C'n}{n^{1.4}}+ O\left(\frac{1}{n^{0.05}}\right) \to 0.
    \end{align*}
    This completes the proof.
\end{proof}
\section{Open problems}\label{sec:open}
In this article, we worked with i.i.d.\ Exponential$(1)$ distributed random variables. We believe \Cref{thm:main} to be true for any continuous density with positive support.
\begin{conjecture}
    \Cref{thm:main} is true for i.i.d.\ weights $(W_{\vec e})_{\vec e \in \vec E}$ with continuous distributions supported in $(0,\infty)$. \end{conjecture}
It is also a natural research direction to extend \Cref{thm:main} to other classes of graphs. We plan to tackle the following question for future works.
\begin{question}\label{question:extension}
    Naturally extend \Cref{thm:main} to other classes of graphs with i.i.d.\ continuously distributed weights with positive support (start with Exponential$(1)$ weights). Some natural candidates for such a sequence $(G_n)_{n \ge 1}$ are:
    \begin{enumerate}[a.]
    \item $G_n$ is the largest connected component in Erd\"os--R\'enyi graphs $G(n,p)$, both the dense setting ($p\in (0,1)$) and sparse setting $(p=\lambda/n) $ where $\lambda>1$. 
        \item $G_n$ is a sequence of tori in $\Z^d$ with volume $n^d$.
        \item $G_n$ is a  $d$-regular expander with $n$ vertices.  
        \item $G_n$ is a  $d$-regular tree of height $n$ (height is the maximal graph distance from the root) where the leaves are all `glued' or `wired' together.
    \end{enumerate}
\end{question}
We refer to \cite{expander_HLW} for a comprehensive survey of expander graphs and their properties and to \cite{bollobas01,janson_RG} for a survey of Erd\"os--R\'enyi random graphs.
For item a. in \Cref{question:extension}, the dense setting should be easier to handle than the sparse setting. In the sparse setting the largest component (called the giant component) has a geometry similar to an expander (item c.).

The existence of local limit for item d. was proved in \cite{RS24} for i.i.d.\ Exponential$(1)$ weights, but a more detailed description of the limit is unknown.  There is a natural class of $d$-regular expanders which converge locally to the infinite $d$-regular tree.
A natural question is whether the limit of the MSA in such $d$-regular expanders is the same as that of the wired $d$-regular tree. We conjecture this is the case.
\begin{conjecture}
    Let $G_n$ be an expander sequence converging locally to the $d$-regular tree. Then the local limit of item c. and d. in \Cref{question:extension} matches. 
\end{conjecture}
We believe item b. is more challenging and extremely interesting even in the i.i.d.\ Exponential$(1)$ case. Simulations show that the CLEB walk in $\Z^2$ is `transient' in the sense that the exposed edges in the contracted graph after $k$-steps ($S_k \cap \vec E_k$ in the language of this article) tend to an infinite path (see \cite[Fig 12]{RS24}). This suggests the existence of a weak limit in $\Z^2$, even in an almost sure sense. We leave this question for future investigation.

\bibliographystyle{abbrv}
\bibliography{MSA}
\end{document}